\documentclass[12pt, reqno]{amsart}

\usepackage{amsfonts}
\usepackage{amssymb}
\usepackage{amsthm}
\usepackage{amsmath}
\usepackage{a4wide}
\usepackage{mathrsfs}
\usepackage{epsfig}
\usepackage{palatino}
\usepackage{tikz,fp,ifthen}
\usepackage{float}
\usepackage{graphicx,cite}
\usepackage{comment}
\usepackage{accents}
\usepackage{xfrac}
\usepackage{bm}
\usepackage{appendix}

\usepackage{marginnote}

\usepackage{graphicx}

\usepackage{color}
\definecolor{citegreen}{rgb}{0,0.6,0}
\definecolor{refred}{rgb}{0.8,0,0}
\usepackage[colorlinks, citecolor=citegreen,linkcolor=refred]{hyperref}

\usepackage{mathtools}
\mathtoolsset{showonlyrefs}

\newtheorem{thm}{Theorem}[section]
\newtheorem{lem}[thm]{Lemma}
\newtheorem{prop}[thm]{Proposition}
\newtheorem{cor}[thm]{Corollary}

\theoremstyle{definition}
\newtheorem{defn}[thm]{Definition}

\theoremstyle{remark}
\newtheorem{rem}[thm]{Remark}

\numberwithin{equation}{section}

\def\cL{\mathcal L}

\def\E{{\sf E}}
\def\R{\mathbb R}

\def\K{\mathbb K}
\def\R{{{\mathbb R}}}

\def\NN{\mathbb N}
\def\N{\mathbb N}

\def\cP{\mathcal P}
\def\cM{\mathcal M}
\def\cB{\mathcal B}
\def\cT{\mathcal T}
\def\cK{\mathcal K}
\def\cA{\mathcal A}
\def\fB{\mathfrak B}
\def\bP{{\sf P}}
\DeclareMathOperator*{\esssup}{ess\,sup}

\newcommand{\W}{\mathcal{W}}
\newcommand{\Norm}[2]{\left\Vert #1 \right\Vert_{#2}}

\usetikzlibrary{backgrounds}
\usetikzlibrary{decorations.pathmorphing,backgrounds,fit,calc,through}
\usetikzlibrary{arrows}
\usetikzlibrary{shapes,decorations,shadows}
\usetikzlibrary{fadings}
\usetikzlibrary{patterns}
\usetikzlibrary{mindmap}
\usetikzlibrary{decorations.text}
\usetikzlibrary{decorations.shapes}

\begin{document}

\title[Well-posedness of KFP equations with unbounded drift]{Well-posedness of Kolmogorov-Fokker-Planck equations with unbounded drift}

\author[F. Anceschi]{Francesca Anceschi}
\address{Francesca Anceschi\\
Dipartimento di Ingegneria Industriale e Scienze Matematiche, Università Politecnica delle Marche, Via Brecce Bianche 12, 
60131 Ancona, Italy}
\email{f.anceschi@staff.unipvm.it}

\author[G. Ascione]{Giacomo Ascione}
\address{Giacomo Ascione\\
Scuola Superiore Meridionale,
Universit\`a di Napoli, Largo San Marcellino 10, 80138 Napoli,
Italy}
\email{g.ascione@ssmeridionale.it}

\author[D. Castorina]{Daniele Castorina}
\address{Daniele Castorina\\
Dipartimento di Matematica e Applicazioni,
Universit\`a di Napoli, Via Cintia, Monte S. Angelo 80126 Napoli,
Italy}
\email{daniele.castorina@unina.it}

\author[F. Solombrino]{Francesco Solombrino}
\address{Francesco Solombrino\\
Dipartimento di Matematica e Applicazioni,
Universit\`a di Napoli, Via Cintia, Monte S. Angelo 80126 Napoli,
Italy}
\email{francesco.solombrino@unina.it}

\begin{abstract}  
We consider Kolmogorov-Fokker-Planck equations with unbounded drift terms which are only measurable in time and locally H\"older continuous in space. In particular, we extend the parametrix method to this setting and we prove existence and uniqueness of measure solutions to the associated Cauchy problem, as well as the equivalence with the corresponding stochastic formulation.
\end{abstract}
\thanks{{FA is supported by GNAMPA-INdAM Project “Problemi non locali: teoria cinetica e non uniforme ellitticità”, CUP$\_$E53C220019320001. GA is supported by the GNAMPA-INdAM Project ''Deterministic Control of Stochastic Dynamics'', CUP$\_$E53C23001670001 and the PRIN 2022XZSAFN Project ''Anomalous Phenomena on Regular and Irregular Domains: Approximating Complexity for the Applied Sciences'' CUP$\_$E53D23005970006. DC and FS are supported by the PRIN 2022HKBF5C Project ``Variational Analysis of Complex Systems in Material Science, Physics and Biology'' CUP$\_$E53D23005720006 . PRIN projects are part of PNRR Italia Domani, financed by European Union through NextGenerationEU. FS additionally acknowledeges support from the GNAMPA-INdAM Project "Problemi di controllo ottimo nello spazio di Wasserstein delle misure definite su spazi di Banach" CUP$\_$E53C23001670001 and from the Starplus 2020 Unina Linea 1, ``New Challenges in the Variational Modeling of Continuum Mechanics" from the University of Naples Federico II and Compagnia di San Paolo.}}

\setcounter{tocdepth}{1}

\maketitle

\section{Introduction} \label{intro}
The aim of this work is to prove existence and uniqueness for measure solutions to the linear Cauchy problem with measure data
\begin{equation}\label{probmuintro}
\begin{cases} \partial_t \mu_t = - v \cdot \nabla_x \mu_t + \sigma \Delta_v \mu_t - \mathrm{div}_v (F(t,x,v) \mu_t) \quad &(t,x,v) \in  (0, T] \times\mathbb{R}^{2d}, \\
\mu_0 = \bar{\mu}  \quad &(x,v) \in \R^{2d},
\end{cases}
\end{equation}
where 
$T,\sigma>0$ (possibly $T=\infty$) and $F:[0,T]\times \R^{2d} \to \R^d$ is a suitable measurable vector field. The assumptions on $F$ will be presented below  and indeed, achieving a well-posedness theory under such assumptions is the main scope of the paper.

\subsection{Presentation of the problem}
Our analysis is focused on three fundamental aspects related to \eqref{probmuintro}. First of all, we are interested in understanding under which conditions problem \eqref{probmuintro} is well-posed. Furthermore, we aim to investigate what are the possible minimal assumptions ensuring the uniqueness of this solution (in a sense later on specified). Lastly, since we are interested in solutions with finite first order moment (hence, in $\W_1$), we 
focus our attention on models where the drift term $F(t,z)$ is possibly unbounded, and which satisfies a suitable growth condition alongside with a local smoothness assumption. Eventually, the equivalence of \eqref{probmuintro}, in the sense of Theorem \ref{thm:existence} below, with a second-order system of SDEs will be the output of our analysis and the starting point for possible applications. These applications, not contained in the present paper, will be shortly presented at the end of this introduction and constitute the scope of further investigation.

\medskip

By duality, the problem of finding a unique solution to \eqref{probmuintro} is connected to the problem of proving the existence of a solution for the dual Kolmogorov equation
\begin{equation} \label{dual-Kolmo}
	\sigma \Delta_v g(t,z) + \partial_t g(t,z) + v \cdot \nabla_x g(t,z) +F(t,z)\cdot \nabla_v g(t,z)= \Psi(t,z), \ (t,z) \in \R^{1+2d},
\end{equation}
where $z=(x,v) \in \R^{2d}$, $\Psi \in C^{\infty}_c(\R^{2d})$ and $\sigma >0$ is the constant appearing in \eqref{probmuintro}. This field of research has an interest of its own, and this problem is classically attacked by proving the existence of a fundamental solution for the operator associated to \eqref{dual-Kolmo} via an adaptation of the Levy parametrix method, see for instance the works \cite{polidoro, francesco2005class} focused on proving the existence of a classical solution to the Cauchy problem associated with \eqref{dual-Kolmo}. Recently, there has been a proliferation of new results extending those for the classical case to Lie solutions (see Section \ref{sub:kfp-lin}) considering diffusion matrices and drift terms with bounded coefficients, only measurable in time and locally H\"older continuous in space, see \cite{lucertinigood, BiagiBramanti}. Nevertheless, to our knowledge the case of unbounded $F$ has never been considered before in this context and for this reason also the technical results of Section \ref{sub:kfp-lin} are a novelty of this work. 

\medskip

Recalling that $z=(x,v) \in \R^{2d}$ and comparing our analysis to the existing literature on the dual equation \eqref{dual-Kolmo}, it is immediately clear that the regularity assumptions for $F$ are anisotropic in the sense that, for any fixed $t \in (0,T)$, the drift term $F$ is assumed to be $\alpha$-H\"older continuous in $v$, and $\frac\alpha3$-H\"older continuous in $x$ (see assumption $(A_1)$). This is due to the underlying non-Euclidean geometry associated with the operator and later on introduced in Section \ref{sub:notation}. 

\medskip

In conclusion, our aim is to develop an existence and uniqueness theory for problems of the type \eqref{probmuintro} with possibly unbounded drift $F$ subject to a growth condition and not necessarily Lipschitz regular in $z$ (in the usual Euclidean sense). Moreover, our analysis provides us with new insights for the study of Kolmogorov equations \eqref{dual-Kolmo} since it extends recent results for the existence of a fundamental solution to the case of unbounded drift terms, that is still not present in the already existing literature.  

\subsection{Statement of main results}
In this work, we are interested in studying the following type of solutions. 
\begin{defn} \label{solution}
	A flow of probability measures $\{\mu_t\}_{t \in [0,T]}\in C([0,T]; \W_1(\R^{2d}))$, where $\W_1(\R^{2d})$ is the $1$-Wasserstein space, is a solution of
	\begin{equation}\label{probmudef}
		\begin{cases} \partial_t \mu_t = - v \cdot \nabla_x \mu_t + \sigma \Delta_v \mu_t - \mathrm{div}_v (F(t,z) \mu_t) \quad &(t,x,v) \in  (0, T] \times\mathbb{R}^{2d}, \\
			\mu_0 = \bar{\mu}  \quad &(x,v) \in \R^{2d},
		\end{cases}
	\end{equation}
	if and only if for all $\psi \in C^\infty_c(\R^{2d})$ and for all $t \in [0,T]$ it holds
	\begin{equation}\label{eq:weaksol}
		\int_{\R^{2d}}\psi \, d\mu_t-\int_{\R^{2d}}\psi \, d\overline{\mu}=\int_0^t \int_{\R^{2d}}\left(v \cdot \nabla_x \psi+ F(s,z)\cdot \nabla_v \psi+ \sigma \Delta_v \psi\right) \, d\mu_s \,  ds. 
	\end{equation}
	If $T=\infty$, we say that $\{\mu_t\}_{t \ge 0}$ is a global solution of \eqref{probmudef}, otherwise we call it a local solution.
\end{defn}
\begin{rem}
	By a density argument, we infer that if $\{\mu_t\}_{t \in [0,T]}$ is a solution of \eqref{probmudef}, then \eqref{eq:weaksol} holds for any $\psi \in {\sf A}(\R^{2d})=\{\psi \in C^1_{\rm b}(\R^{2d}): \ \Delta_v \psi \in C_{\rm b}(\R^{2d})\}$.
\end{rem}
In particular, we are interested in proving the existence of a global solution for \eqref{probmudef} under minimal assumptions, and then we will tackle the uniqueness of such a solution. 
To prove that \eqref{probmudef} admits at least a local solution, we have to make the following \textit{minimal} assumptions on the vector field $F$.
\begin{itemize}
	\item[$(A_0)$] $F:[0,+\infty) \times \R^{2d} \to \R^d$ is a Carath\'{e}odory map, i.e. it is measurable in the first variable and continuous in the second one. \label{itemA0}
	\item[$(A_1)$] For any $T>0$ there exists a constant $C(F)>0$ such that $F(t,z) \le C(F)(1+|z|_B^{\beta})$ for all $z \in \R^{2d}$, with $\beta \in (0,1)$.
	\item[$(A_2)$] For any $T>0$ and any compact set $K \subset \R^{2d}$ there exist two constants $L>0$ and $\alpha\in (\beta,1]$ such that
	\begin{equation*}
		|F(t,z_1)-F(t,z_2)| \le L \Norm{z_1-z_2}{\rm B}^{\alpha}, \ \forall \, z_1,z_2 \in K, \ \forall \,  t \in [0,T].
	\end{equation*}
\end{itemize}
Notice that Assumption $(A_2)$ makes explicit use of the Lie structure we will introduce in Section~\ref{sub:notation}.

To get the existence result, we will use the theory of stochastic differential equations (SDEs) and for this reason we fix a probability space $(\Omega, \Sigma, \bP)$. Let us recall the following definition.
\begin{defn}
	Consider the SDE
	\begin{equation}\label{eq:SDEstrongaux-introdef}
		\begin{cases}
			dX(t)=V(t)\, dt & t \in [0,T]\\
			dV(t)=F(t,X(t),V(t))\, dt+\sqrt{2\sigma} \, dB(t) & t \in [0,T] \\
			X(0)=X_0, \quad V(0)=V_0
		\end{cases}
	\end{equation} 
	where $B$ is a $d$-dimensional Brownian motion on $(\Omega, \Sigma, \bP)$, $F$ satisfies Assumptions $(A_0)$, $(A_1)$ and $(A_2)$, $\sigma>0$, $(X_0,V_0) \in \cM(\Omega;\R^{2d})$ and $T>0$ (possibly $T=\infty$). We say that \eqref{eq:SDEstrongaux-introdef} admits a strong solution if there exists a stochastic process $(X,V)$ such that:
	\begin{itemize}
		\item[$(i)$] $\int_0^tV(s)\, ds$ and $\int_{0}^t F(s,X(s),V(s))\, ds$ are well-defined for any $t \in [0,T]$;
		\item[$(ii)$] For any $t \in [0,T]$ we have
		\begin{equation*}
			V(t)=V_0+\int_0^tF(s,X(s),V(s))\, ds+\sqrt{2\sigma}B(t), \qquad X(t)=X_0+\int_0^t V(s)\, ds.
		\end{equation*}
	\end{itemize}
	The process $Z=(X,V)$ is the strong solution of \eqref{eq:SDEstrongaux-introdef}. We say that uniqueness in law holds for the SDE \eqref{eq:SDEstrongaux-introdef} if for any two solutions $Z_1,Z_2$ one has ${\rm Law}(Z_1(t))={\rm Law}(Z_2(t))$ for any $t \in [0,T]$. Furthermore, we say that pathwise uniqueness holds if $Z_1=Z_2$ almost surely (a.s.).
\end{defn}

On the other hand, uniqueness of the solution of \eqref{eq:weaksol} is shown by means of a duality approach based on the results of Section \ref{sub:kfp-lin}. The following global existence and uniqueness theorem holds true.
\begin{thm}\label{thm:existence}
	Under Assumptions $(A_0)$,$(A_1)$ and $(A_2)$, for any $\mu_0 \in \W_1(\R^{2d})$ and $T>0$, the equation
	\begin{equation}\label{probmu}
		\begin{cases} \partial_t \mu_t = - v \cdot \nabla_x \mu_t + \sigma \Delta_v \mu_t - \mathrm{div}_v (F(t,z) \mu_t) \quad &(t,x,v) \in  \R^+_0 \times\mathbb{R}^{2d}, \\
			\mu_0 = \bar{\mu}  \quad &(x,v) \in \R^{2d},
		\end{cases}
	\end{equation}
	admits exactly one global solution ${\bm \mu} :=  \{\mu_t\}_{t \in \R_0^+}\in C(\R_0^+; \W_1(\R^{2d}))$, which can be characterized as follows. Let $B$ be a $d$-dimensional Brownian motion and $(X_0,V_0) \in \cM(\Omega;\R^{2d})$ be independent of $B$ and such that ${\rm Law}(X_0,V_0)=\mu_0$, then there exists a stochastic process $(X,V) \in L^1(\Omega; C([0,T]; \R^{2d}))$ for all $T>0$ such that ${\bm \mu} = Law(X,V)$ and $(X,V)$ is the pathwise unique strong solution to 
	\begin{equation}\label{eq:SDEstrongaux-intro}
		\begin{cases}
			dX(t)=V(t)\, dt & t \in \R_0^+\\
			dV(t)=F(t,X(t),V(t))\, dt+\sqrt{2\sigma} \, dB(t) & t \in \R_0^+ \\
			X(0)=X_0, \quad V(0)=V_0.
		\end{cases}
	\end{equation} 
\end{thm}


	

\begin{rem}
	In particular, we can say that for any $T>0$ the unique solution of \eqref{probmu} is the flow of probability measures obtained by applying the evaluation map on a measure $\bm{\mu} \in \W_1(C([0,T];\R^{2d}))$.
	
	It is also worth noticing that if $\sigma=0$, then Assumptions $(A_0)$, $(A_1)$ and $(A_2)$ do not guarantee the uniqueness of the solution for the system of Ordinary Differential Equations \eqref{eq:SDEstrongaux-intro}. From this point of view, Theorem \ref{thm:existence} can be seen as a regularization by noise result, in a degenerate noise setting (see \cite{Flandoli11} and references therein).
\end{rem}

\subsection{Further directions} The well-posedness results descibed above are in the present paper recovered for the case of a linear Kolmogorov equation. Indeed, they constitute in our opinion the starting point of a broader investigation, concerning particle systems of the second order whose prototypical agent is driven by a McKean-Vlasov SDE, or by a Vlasov-Fokker-Plank PDE. In this case the drift field depends on the solution $\mu$ itself in a nonlocal way, a typical example being the choice 
\begin{equation}\label{introeq:nonlocaldrift}
F(t, \mu, z)=H(t,z)\ast \mu_t\,.
\end{equation}
A future objective is extending the results obtained here to such a case, which is relevant from the point of view of the applications to mean-field control problems for multi-agent systems. The (simpler) case of first order systems has been addressed, also with well-posedness results for a class of optimal control problems, in a number of recent papers (we may refer for instance to \cite{Ascione20236965, Orlando2023}, where also applications are presented.) 

Even if in literature boundedness and global Lipschitz continuity for the kernel $H$ in \eqref{introeq:nonlocaldrift} are often  assumed, also  in the nonlinear case we strive to obtain results under weaker assumptions (growth conditions at infinity in place of global  boundedness, local H\"older continuity etc.), exactly in the spirit of the present paper. Moreover, we remark, as it can also be inferred by related first-order cases as in \cite{Ascione20236965, Almi20231},  that the equivalence with a stochastic formulation is particularly useful for uniqueness issues, fixed point arguments, and stability estimates on the solutions which allow one to prove well-posedness of the related optimal control problems. This research program will be carried out in some forthcoming papers.

\subsection{Plan of the work}
In Section \ref{sub:notation} we introduce the necessary notation and some preliminary results useful for our work. Section \ref{sub:kfp-lin} is devoted to 
the proof of existence results for the dual Kolmogorov equation \eqref{dual-Kolmo}, through an extension of the parametrix method for the unbouded drift case. The proof of our main results is exploited in Section \ref{proofthm}. Finally, Appendix \ref{app:A} contains useful results regarding group convolutions in the kinetic setting.

\section{Useful notation and preliminary results} \label{sub:notation}
We denote by $\cM(\Omega;\fB)$ the space of $\fB$-valued random variables (i.e. Borel-measurable functions $X:\Omega \to \fB$). The law of a random variable $X \in \cM(\Omega;\fB)$ is defined as ${\rm Law}(X)=X \sharp 
\bP$, where $\sharp$ denotes the pushforward, i.e. for any $A \in \cB(\fB)$
\begin{equation*}
	{\rm Law}(X)(A)=\bP(X^{-1}(A)).
\end{equation*}
We denote by $L^p(\Omega; \fB)$ the subspace of $\cM(\Omega;\fB)$ of random variables $X$ such that
\begin{equation*}
	\E[(d_{\fB}(X,x_0))^p]<\infty, \ \forall x_0 \in \fB,
\end{equation*}
where $\E$ is the expected value.
In this setting, we can define a coupling of two probability measures $\mu,\nu \in \cP(\fB)$ as any random variable $(X,Y) \in \cM(\Omega;\fB \times \fB)$ such that ${\rm Law}(X)=\mu$ and ${\rm Law}(Y)=\nu$ (see \cite[Definition 1.1]{villani}). 

For any $T>0$, we denote by $C([0,T];\fB)$ the space of continuous functions $f:[0,T] \to \fB$ equipped with the uniform distance, i.e.
\begin{equation*}
	d_{\infty}(f,g)=\sup_{t \in [0,T]}d_{\fB}(f(t),g(t)).
\end{equation*}
Clearly, $C([0,T];\fB)$ is a complete separable metric space.

%
%

\subsection{Flows of probability measures}
For reader's convenience, in the following we will use the notation $\bm{\mu} \in C([0,T];\W_p(\mathfrak{B}))$ to also denote continuous flows of probability measures, with $\bm{\mu}=\{\mu_t\}_{t \in [0,T]}$. We write $\bm{\mu} \in \W_p(C([0,T];\mathfrak{B}))$ to state that there exists a Borel probability measure (that we still denote $\bm{\mu}$) on $C([0,T];\mathfrak{B})$ such that $\mu_t={\rm ev}_t \sharp \mu$. With this identification, we can say that $\W_p(C([0,T];\mathfrak{B}))\subset C([0,T];\W_p(\mathfrak{B}))$. Furthermore, for any $T_1>T_2$ and any $\bm{\mu}=\{\mu_t\}_{t \in [0,T_1]} \in C([0,T_1];\W_p(\mathfrak{B}))$ we still denote $\bm{\mu}=\{\mu_t\}_{t \in [0,T_2]}$ (i.e. truncating the flow at $T_2$), so that through this choice we have $C([0,T_1];\W_p(\mathfrak{B})) \subset C([0,T_2];\W_p(\mathfrak{B}))$.

We will denote by $C([0,T];\cP(\fB))$ the space of narrowly continuous flows of probability measures, i.e. such that for any $t_0 \in [0,T]$ and any $f \in C_{\rm b}(\fB)$ it holds
\begin{equation*}
	\lim_{t \to t_0} \int_{\fB} f(x)d\mu_t(x)=\int_{\fB} f(x)d\mu_{t_0}(x),
\end{equation*}
see \cite[Section 5.1]{ambrosio2005gradient}. 

\subsection{Random operators}
Let $\fB$ be a Banach space. A function $\cT:\Omega \times \fB \to \fB$ is said to be a random operator if, for any $x \in \fB$, $\cT(\cdot,x) \in \cM(\Omega;\fB)$. We say that a random operator is continuous if, for any $\omega \in \Omega$, $\cT(\omega,\cdot):\fB \to \fB$ is continuous. A random operator is compact if, for any $\omega \in \Omega$, $\cT(\omega,\cdot):\fB \to \fB$ is a compact operator. A random fixed point of random operator $\cT$ is a random variable $X \in \cM(\Omega;\fB)$ such that for any $\omega \in \Omega$ we have $\cT(\omega,X(\omega))=X(\omega)$, in shorthand notation $\cT X=X$. We recall here the Bharucha-Reid fixed point theorem (see \cite[Corollary 2.2]{itoh1979random}), which is a random version of Schauder's fixed point theorem.
\begin{thm}\label{thm:BRfp}
	Let $\mathfrak{X}$ be a closed convex subset of $\fB$ and $\cT:\Omega \times \mathfrak{X} \to \mathfrak{X}$ a continuous compact random operator. Then there exists at least one random fixed point.
\end{thm}

\subsection{H\"older spaces}
In the following, we denote the elements of $\R^{2d}$ either by $z \in \R^{2d}$ or by $(x,v) \in \R^{d}\times \R^d$.
As pointed out for the first time in \cite{LP}, the correct 
geometrical framework for the analysis of our problem is a Lie group. For this reason, we endow~$\mathbb{R}^{1+ 2d}$ with the group law 
\begin{equation}\label{group_law}
	(t,z) \mapsto (t_0,z_0) \circ (t,z) = (t_0+t, x_0+x+tv_0,  v_0+v), 
\end{equation}
for every $(t_0,z_0)=(t_0,x_0,v_0), (t,z)=(t,x,v) \in \mathbb{R}^{1+2d}$.
Then~$\mathbb{K}:=(\mathbb{R}^{1+2d}, \circ)$ is a Lie group with identity element~$e:=(0,0,0)$ and inverse defined by
$$
(t,z)^{-1}:=(-t, -x+tv,-v), \qquad \forall \, (t,z)=(t,x,v) \in \mathbb{R}^{1+2d}.
$$
Additionally, we introduce a suitable dilation group~$\{\Phi_{r}\}_{r>0}$ defined as
\begin{equation}\label{dilation}
	\Phi_r : \mathbb{R}^{1+2d} \mapsto \mathbb{R}^{1+ 2d}, \quad \Phi_r (t,z):=(r^2 t, r^3 x, r v), \quad \forall r >0,
\end{equation}
which respects the intrinsic scaling of problem \eqref{probmu}.
Hence, from now on we shall consider the homogeneous norm in $\mathbb{R}^{1+2d}$
associated to the group of dilations~$\Phi_r$ defined in~\eqref{dilation} and given by
\begin{equation} \label{groupnorm}
	\|  \cdot \|_{\mathbb{K}} : \mathbb{R}^{1+2d}\to [0,\infty), 
	\quad \| (t,z) \|_{\mathbb{K}}:=  |t |^{\frac{1}{2}} +  \| z \|_{\rm B} :=
	|t |^{\frac{1}{2}} +  \sum_{i=1}^d | x_i |^{\frac{1}{3}} + \sum_{i=1}^d  | v_i |.
\end{equation}

Starting from the traslation \eqref{group_law} and dilation \eqref{dilation} group, we are allowed to introduce a left-invariant $1$-homogeneous quasi-distance on $\mathbb{R}^{2d+1}$ defined as
\begin{align*}
	d_{\K}(t,z; \tau, \zeta) = \| (\tau, \zeta)^{-1} \circ (t,z) \|_{\mathbb{K}}.
\end{align*}
The quasi-distance $d$ is globally equivalent to the control distance of the group and, being a quasi-distance, there exists a structural constant $k>0$ such that
\begin{align*}
	&d_{\K}(t,z; \tau, \zeta) \le k \left(d_{\K}(t,z; s,y) + d_{\K}(s,y; \tau, \zeta) \right) 
	\qquad \, \, \forall (\tau, \zeta),(t,z),(s,y) \in \mathbb{R}^{1+2d}; \\
	&d_{\K}(t,z; \tau, \zeta) \le k \, d_{\K}( \tau, \zeta; t, z) \qquad \qquad \quad \qquad \qquad  \, \, \, \, \forall (\tau, \zeta)\in \mathbb{R}^{1+2d}.
\end{align*}
Furthermore, we are allowed to introduce the associated $d_{\K}$-balls defined as
$$
{\sf B}_r(t_0, z_0):= \{(t,z) \in \mathbb{R}^{1+2d}: d_{\K}(t_0, z_0;t,z)<r\},
$$ 
and, thanks to Remark 1.1 of \cite{BiagiBramanti}, we observe 
\begin{align*}
	|{\sf B}_r (\tau, \zeta)| = |{\sf B}_r(0)| = \omega r^{2+ 4d} ,
\end{align*}
where $\omega := |{\sf B}_1(0)| >0$ and $2+ 4d$ is the 
homogeneous dimension of $\mathbb{R}^{1+ 2d}$.

Finally, we recall the definition of anisotropic H\"older spaces $C^\alpha_{\rm B}(\R^{2d};\R^m)$, where \linebreak $m \in \NN$ and $\alpha \in (0,3]$, as given in \cite[Definition $1.1$]{lucertinigood}.
\begin{defn}
	Let $\alpha \in (0,1]$. We define
	\begin{align} \label{lip-b}
		C^\alpha_{\rm B} (\R^{2d};\R^m) &:= \left\{ \psi_{p,M} \in C_{\rm b} (\R^{2d};\R^m) \, : \, \sup_{z \in \R^{2d}} |\psi(z)|+ \sup_{z_1, z_2 \in \R^{2d}} \frac{|\psi(z_2)-\psi(z_1)|}{\|z_2-z_1\|_{\rm B}^\alpha} < \infty	\right\}.
	\end{align}
	For $\alpha \in (1,2]$, we define $C^\alpha_{\rm B} (\R^{2d};\R^m)$ as the space of functions $\psi:\R^{2d} \to \R^m$ such that:
	\begin{itemize}
		\item[$(i)$] $\psi \in C_{\rm b} (\R^{2d};\R^m)$;
		\item[$(ii)$] $\psi$ is differentiable in $v_1,\dots,v_d$ with $\nabla_v \psi \in C^{\alpha-1}_{\rm B}(\R^{2d};\R^{m \times d})$;
		\item[$(iii)$] It holds
		\begin{equation*}
			\sup_{\substack{x,h,v \in \R^{2d} \\ h \not = 0 }}\frac{|\psi(x+h,v)-\psi(x,v)|}{|h|^{\frac{\alpha}{3}}}<\infty.
		\end{equation*}
	\end{itemize}
	We use the same definition for $\alpha \in (2,3]$: in such a case, this implies that the Hessian tensor $\nabla^2_v \psi$ is well-defined and belongs to $C^{\alpha-2}_{\rm B}(\R^{2d};\R^{m \times d \times d})$. If $m=1$, we directly use the notation $C_{\rm B}^\alpha(\R^{2d})$.\\
\end{defn}

For further information on Lie groups associated to kinetic equations we refer the reader to the 
\cite{AP-survey,BLU, BiagiBramanti}, and the references therein.

\subsection{Lie differentiability}\label{sec:Liediff}
It is known (see for instance \cite{lucertinigood}) that we cannot expect the solutions of \eqref{probmu} to be regular separately in the $t$ and $x$ variables. In order to handle this regularity issue, we need to exploit the Lie structure introduced in the previous subsection. In the following, we denote by $\mathbb{O}_d$ and $\mathbb{I}_d$ respectively the $d \times d$ null and identity matrices. We then introduce the $2d \times 2d$ constant matrix $B$ and the exponential of the group 
\begin{align} \label{matrixB}
	B = \begin{pmatrix}
		\mathbb{O}_d & \mathbb{I}_d \\
		\mathbb{O}_d & \mathbb{O}_d
	\end{pmatrix},
		\qquad 
			e^{\tau B} = \mathbb{I}_{2d} + \tau B = \begin{pmatrix}
			\mathbb{I}_d& \tau \mathbb{I}_d \\
			\mathbb{O}_d & \mathbb{I}_d 
		\end{pmatrix},
		\qquad \forall \tau > 0.
\end{align}
\begin{rem}
	We adopt the ordering $(x,v)$ in place of $(v,x)$ used in \cite{lucertinigood}. For this reason, $B$ is assumed to be upper triangular instead of lower triangular.
\end{rem}

For $f \in C ([0,T)\times \R^{2d})$ we define the Lie derivative in $(t,z) \in [0,T) \times \R^{2d}$ as
\begin{equation}\label{eq:Ydef}
	Yf(t,z) := \lim_{s\to t} \frac{f(s,e^{(s-t)B}z)- f(t,z)}{s-t},
\end{equation}
provided that the limit exists. In such a case, we say that $f$ is Lie differentiable in $(t,z)$.
\begin{rem}\label{YC1}
	It is readily seen that $Yf= (v\cdot\nabla_x  + \partial_t) f$ when $f \in C^1 ((0,T)\times \R^{2d})$: indeed, we just have to observe that $s \mapsto (s, e^{(s-t)B}z)$ is the integral curve of the \textit{drift term} $v\cdot\nabla_x  + \partial_t$.
\end{rem}
Thanks to the previous remark, we can also introduce a notion of Lie differentiability which is analogous to absolute continuity in the Euclidean setting, as introduced in \cite{lucertinigood}.
\begin{defn}
	Let $I \subseteq [0,+\infty)$ be an interval. We denote by ${\rm AC}_Y(I;C_{\rm b}(\R^{2d}))$ the space of continuous functions $f:I\times \R^{2d} \to \R$ such that $f(t,\cdot) \in C_{\rm b}(\R^{2d})$ for any $t \in I$ and there exists a function $Yf \in L^1_{\rm loc}(I;C_{\rm b}(\R^{2d}))$ with the property that for any $s,t \in I$
	\begin{equation}\label{eq:ACYf}
		f(s,e^{(s-t)B}z)=f(t,z)+\int_t^s Yf(\tau,e^{(\tau-t)B}z)d\tau,
	\end{equation}
	where we recall the definition of Bochner-Lebesgue space $L^1_{\rm loc}(I;C_{\rm b}(\R^{2d}))$ as the space of all measurable functions $g:I \to C_{\rm b}(\R^{2d})$ such that for any compact $K \subset I$ it holds
	\begin{equation*}
		\int_{K}\Norm{g(t)}{L^\infty(\R^{2d})}dt<\infty.
	\end{equation*}
	In such a case, $Yf$ is the a.e. Lie derivative of $f$. Clearly, if $I=[0,T]$, then $Yf \in L^1(I;C_{\rm b}(\R^{2d}))$.
\end{defn}
\begin{rem} \label{remark-Y}
	Notice that we still use the notation $Yf$ in the previous definition. Indeed, if $f \in {\rm AC}_Y(I;C_{\rm b}(\R^{2d}))$, then the limit in \eqref{eq:Ydef} exists for a.e. $(t,z) \in I \times \R^{2d}$ and coincides with the a.e. Lie derivative $Yf$. In the usual Euclidean case, this is similar to the request that $f$ is absolutely continuous. However, due to the fact that we are implicitly considering $f:I \to C_{\rm b}(\R^{2d})$ and the Banach space $C_{\rm b}(\R^{2d})$ does not satisfy the Radon-Nykodim property (see \cite[Definition 1.2.5 and Example 1.2.8.a]{arendt2011vector}), we would be asking for a stronger property that implies absolute continuity of $C_{\rm b}(\R^{2d})$-valued functions in the Euclidean setting (see \cite[Propositions 1.2.2 and 1.2.3]{arendt2011vector}). From this point of view, the space ${\rm AC}_Y(I;C_{\rm b}(\R^{2d}))$ does not characterize \textit{absolute continuous functions with respect to the Lie derivative $Y$}, but actually, according to \cite[Proposition 1.2.3]{arendt2011vector}, \textit{$C_{\rm b}(\R^{2d})$-valued absolute continuous functions (on the integral curve of $Y$) that are a.e. Lie differentiable.} 
\end{rem}

We may also consider H\"older spaces with respect to the integral curves of $Y$, as done in \cite[Definition 1.2]{lucertinigood}. 
\begin{defn}
	Let $\alpha \in (0,1]$ and $T>0$. We denote by $C_Y^\alpha((0,T)\times \R^{2d})$ the space of functions $f:(0,T)\times \R^{2d} \to \R$ such that
	\begin{equation*}
		\Norm{f}{C_Y^\alpha((0,T)\times \R^{2d})}:=\sup_{\substack{t,s \in (0,T), \ z \in \R^{2d} \\ t \not = s }}\frac{|f(s,e^{(s-t)B}z)-f(t,z)|}{|t-s|^\alpha}<\infty.
	\end{equation*}
\end{defn}
Combining the H\"older regularity with respect to $Y$ with the anisotropic H\"older spaces previously introduced in this section, we provide the following definition of intrinsic H\"older spaces, as done in \cite[Definition 1.3]{lucertinigood}.
\begin{defn}
	Let $\alpha \in (0,3]$ and $T>0$. We define the intrinsic H\"older space $C_{{\rm B},Y}^{\alpha}((0,T)\times \R^{2d})$ as follows
	\begin{itemize}
		\item[$(i)$] If $\alpha \in (0,1]$, then 
		\begin{equation*}
			C_{{\rm B},Y}^{\alpha}((0,T)\times \R^{2d})=L^\infty((0,T);C_{\rm B}^\alpha(\R^{2d})) \cap C_{Y}^{\alpha/2}((0,T)\times \R^{2d}),
		\end{equation*}
		where we recall the definition of the Bochner-Lebesgue space $L^\infty((0,T);C_{\rm B}^\alpha(\R^{2d}))$ as the space of measurable functions $f:(0,T)\to C_{\rm B}^\alpha(\R^{2d})$ such that
		\begin{equation*}
			\Norm{f}{L^\infty((0,T);C_{\rm B}^\alpha(\R^{2d}))}:=\esssup_{t \in (0,T)}\Norm{f(t)}{C_{\rm B}^\alpha(\R^{2d})}<\infty.
		\end{equation*}
		Furthermore, we define
		\begin{equation*}
			\Norm{f}{C_{{\rm B},Y}^\alpha((0,T)\times \R^{2d})}:=\Norm{f}{L^\infty((0,T);C_{\rm B}^\alpha(\R^{2d}))}+\Norm{f}{C_Y^{\alpha/2}((0,T)\times \R^{2d})}.
		\end{equation*}
		\item[$(ii)$] If $\alpha \in (1,2]$, then $f \in C_{{\rm B},Y}^{\alpha}((0,T)\times \R^{2d})$ if and only if $f \in L^\infty((0,T);C_{\rm B}^\alpha(\R^{2d})) \cap C_{Y}^{\alpha/2}((0,T)\times \R^{2d})$ and $\nabla_v f \in C_{{\rm B},Y}^{\alpha-1}((0,T)\times \R^{2d})$, where $\nabla_v f$ is well-defined since $f(t,\cdot) \in C_{\rm B}^{\alpha}(\R^{2d})$ for a.a. $t \in (0,T)$. Furthermore, we define
		\begin{equation*}
			\Norm{f}{C_{{\rm B},Y}^\alpha((0,T)\times \R^{2d})}:=\Norm{f}{L^\infty((0,T);C_{\rm B}^\alpha(\R^{2d}))}+\Norm{\nabla_v f}{C_{{\rm B},Y}^{\alpha-1}((0,T)\times \R^{2d})}+\Norm{f}{C_Y^{\sfrac{\alpha}{2}}((0,T)\times \R^{2d})}.
		\end{equation*}
		\item[$(iii)$] If $\alpha \in (2,3]$, then $f \in C_{{\rm B},Y}^{\alpha}((0,T)\times \R^{2d})$ if and only if $f \in L^\infty((0,T);C_{\rm B}^\alpha(\R^{2d}))$, $\nabla_v f \in C_{{\rm B},Y}^{\alpha-1}((0,T)\times \R^{2d})$ and $Yf \in L^\infty((0,T);C_{\rm B}^{\alpha-2}(\R^{2d}))$. Furthermore, we define
		\begin{equation*}
			\Norm{f}{C_{{\rm B},Y}^\alpha((0,T)\times \R^{2d})}:=\Norm{f}{L^\infty((0,T);C_{\rm B}^\alpha(\R^{2d}))}+\Norm{\nabla_v f}{C_{{\rm B},Y}^{\alpha-1}((0,T)\times \R^{2d})}+\Norm{Yf}{L^\infty((0,T);C_{\rm B}^{\alpha-2}(\R^{2d}))}.
		\end{equation*}
	\end{itemize}
\end{defn}

\section{Existence results for the dual Kolmogorov equation}
\label{sub:kfp-lin}
In this section we prove some auxiliary results regarding existence and uniqueness of strong 
Lie solutions to a generalized version of the adjoint problem associated to \eqref{probmu}, that is 
\begin{equation}\label{eq:adjeq-app}
	\begin{cases}
		\cL u(t,z)=\Psi(t,z) &(t,z) \in [0,T) \times \R^{2d} \\
		u(T,z)=g(z) & z \in \R^{2d}
	\end{cases}
\end{equation}
where $T>0$, $\Psi:[0,T) \times \R^{2d} \to \R$ and $g:\R^{2d} \to \R$ are suitable measurable functions.
In particular, the adjoint Kolmogorov operator $\cL$ is defined as
\begin{equation*}
	\cL = \mu \Delta_v +Y+F(t,z)\cdot \nabla_v, \ (t,z) \in \R^{1+2d},
\end{equation*}
where we recall that the Lie derivative $Y$ is defined in \eqref{eq:Ydef}, $z=(x,v) \in \R^{2d}$ and $\mu >0$.
For simplicity, let us introduce the notation
\begin{equation*}
	\cA_{\Psi}u(t,z)= 
	\mu \Delta_v +F(t,z)\cdot \nabla_vu(t,z)-\Psi(t,z), \ (t,z) \in \R^{1+2d},
\end{equation*}
so that the equation can be rewritten as $Yu=-\cA_{\Psi}u$. 
\begin{defn} \label{def-sol-3}
	A continuous function $u:[0,T] \times \R^{2d} \to \R$ is said to be a strong Lie solution of \eqref{eq:adjeq-app} if and only if
	\begin{itemize}
		\item[$(i)$] $\partial_{v_i}u, \partial^2_{v_i\, v_j}u \in L^1_{\rm loc}([0,T); C_{\rm b}(\R^{2d}))$ for $i,j=1,\dots,d$;
		\item[$(ii)$] for any $(t,z) \in \R^{2d}$ we have
		\begin{equation*}
			u(t,z)=g(z)+\int_t^T \cA_{\Psi}u(\tau,e^{(\tau-t)B}z)d\tau.
		\end{equation*}
	\end{itemize}
\end{defn}
\begin{rem}\label{rem:pointwise2}
	Such definition is equivalent to the one given in \cite[Definition 1.5]{lucertinigood} once the final condition $u(T,\cdot)\equiv 0$ is prescribed. Furthermore, if $u$ is a strong Lie solution of \eqref{eq:adjeq-app}, then $Yu$ is defined a.e. and we have $Yu \in L^1_{\rm loc}([0,T); C_{\rm b}(\R^{2d}))$, i.e. $u \in {\rm AC}_Y([0,T);C_{\rm b}(\R^{2d}))$, see Remark \ref{remark-Y}. Finally, observe that $u$ satisfies \eqref{eq:adjeq-app} pointwise for a.e. $(t,z) \in (0,T)\times \R^{2d}$ and then it is a classical pointwise solution if $\cA_{\Psi}u$ is continuous.
\end{rem}

As already stated in the introduction of this work, the problem of proving the existence and uniqueness of a solution to \eqref{eq:adjeq-app} originated from 
\cite{polidoro} and its study under mild assumptions for $F$ has been of great interest for
the community in recent years, see for instance the recent works \cite{BiagiBramanti, lucertinigood}. Still, to the best of our knowledge, there is no available result considering assumptions
$(A_0)$, $(A_1)$ and $(A_2)$ for the drift $F$, even in the case of constant coefficients for the matrix $A$. 
In order to reach our goal, we implement an extension of the parametrix method given in \cite{lucertinigood} adapting an idea already 
considered in \cite{DeckKruse} for parabolic equations to the hypoelliptic setting. 
For this purpose, we define the principal part operator $\cK$ of $\cL$ as
\begin{equation*}
	\cK_\lambda= \lambda \Delta_v  +Y, \ (t,z) \in \R^{1+2d}.
\end{equation*}
This is a constant coefficients operator, sharing the same transport term of the operator $\cL$ and with no lower order terms. 
Furthermore, 
it is homogeneous with respect to the group dilations $\Phi_r$, i.e. if $f$ is a solution of $\cK_\lambda f \equiv 0$, 
then the function $f_r(t,z)=f(r^2t,r^3x,rv)$ satisfies the same equation on a rescaled domain. 
Analogously, $f_{(t_0,z_0)}(t,z)=f((t_0,z_0)\circ (t,z))$ still satisfies $\cK_{\lambda}f_{(t_0,z_0)} = 0$.
Additionally, if we define $V_{i,\lambda}:=\sqrt{\lambda}\partial_{v_i}$, we can rewrite
\begin{equation*}
	\cK_\lambda:=\sum_{i=1}^{d}V_{i,\lambda}^2+Y.
\end{equation*}
In particular $\cK_\lambda$ satisfies the H\"ormander condition, i.e. the Lie algebra generated by $V_{i,\lambda}, [V_{i,\lambda},Y]$ for $i=1,\dots,d$, 
denoted by ${\sf Lie}\{V_{i,\lambda},[V_{i,\lambda},Y], \ i=1,\dots,d\}$, is of full rank $1+2d$, where the definition of Lie brackets is $[X,Y]=XY-YX$.

\medskip

Now, we are in a position to recall the definition of fundamental solution of the operator $\cL$ according to \cite[Definition 1.6]{lucertinigood}.
\begin{defn} \label{def-fs2}
	A fundamental solution of $\mathcal{L}$ on $[0,T]\times \R^{2d}$ is a function $p=p(s,z; t, y)$ defined for $0 \le s < t \le T$ and $z,y \in \R^{2d}$ such that:
	\begin{enumerate}
		\item $p(\cdot, \cdot; t, y)$ is a solution of the equation $\mathcal{A}_{\Psi} u + Y u=0$ on $(0,T) \times \R^{2d}$ in the sense of Definition \ref{def-sol-3};
		\item for every $g \in C_{\rm b}(\R^{2d})$ we have
		\begin{equation*}
			\lim_{\substack{(s,z) \to (t,y) \\ s<t}}\int_{\R^{2d}} p(s,z; t, \eta) g(\eta) \, d\eta = g(y),
		\end{equation*}
		i.e. the family of measures $\{p(s,z;t,\eta)d\eta\}_{(s,z) \in [0,t) \times \R^{2d}}$ converges narrowly to $\delta_{y}$ as $(s,z) \to (t,y)$.
	\end{enumerate}
\end{defn}

\medskip

Analogously, for $s<t$ and $z,y \in \R^{2d}$ one defines the fundamental solution of $\cK_\lambda$, that from now on will be denoted as ${\bm P}^{\lambda}(s,z;t,y)$.
More specifically, given $\cK_{\lambda}$ is a constant coefficient operator, we are able to explicitly write the expression of its fundamental solution ${\bm P}^{\lambda}$. Indeed, if we define the covariant matrix associated to $B$, see \eqref{matrixB}, as
\begin{align*}
	\mathcal{C}_\lambda(t)&=\lambda \int_0^te^{\tau B}\begin{pmatrix}
		\mathbb{O}_d & \mathbb{O}_d \\ \mathbb{O}_d & \mathbb{I}_d \end{pmatrix} e^{\tau B^\ast}d\tau=\lambda \begin{pmatrix}
		\displaystyle \frac{t^3}{3}\mathbb{I}_d & \displaystyle \frac{t^2}{2}\mathbb{I}_d \\[10pt] \displaystyle \frac{t^2}{2}\mathbb{I}_d & \displaystyle t \mathbb{I}_d
	\end{pmatrix},
\end{align*}
then one can explicitly compute its determinant and its inverse as
\begin{equation*}
	{\sf det}(\mathcal{C}_\lambda(t))=\frac{\lambda^d t^{4d}}{(12)^d}, \qquad \mathcal{C}_\lambda^{-1}(t)=\frac{1}{\lambda^d}\begin{pmatrix} \displaystyle \frac{12}{t^3}\mathbb{I}_d & \displaystyle -\frac{6}{t^2}\mathbb{I}_d \\[10pt]
		\displaystyle -\frac{6}{t^2}\mathbb{I}_d & \displaystyle \frac{4}{t}\mathbb{I}_d \end{pmatrix} ,
\end{equation*}
so that the fundamental solution ${\bm P}^\lambda$ can be expressed as
\begin{equation}\label{eq:fun-sol-const}
	{\bm P}^\lambda(s,z;t,y)=\frac{\exp\left(-\frac{1}{2}(y-e^{(t-s)B}z)^{\sf t}\mathcal{C}_\lambda^{-1}(t-s)(y-e^{(t-s)B}z)\right)}{\sqrt{(2\pi)^{2d}{\sf det}(\mathcal{C}_\lambda(t-s))}},
\end{equation}
where with ${}^{\sf t}$ we denote the transpose of a matrix (or vector). In particular, for any fixed $s,t \in [0,T]$ with $s<t$ and $z \in \R^{2d}$, we have
\begin{equation}\label{eq:integ1Gamma}
	\int_{\R^{2d}}{\bm P}^\lambda(s,z;t,y)\, dy=1.
\end{equation}
Furthermore, it is not difficult to check that
\begin{equation}\label{eq:integ1Gam2}
	\int_{\R^{2d}}{\bm P}^\lambda(s,z;t,y)\, dz=1.
\end{equation}
Indeed, if we write the integral explicitly and we use the change of variables \linebreak $\zeta=-e^{(t-s)B}z$ (whose absolute value of the Jacobian determinant is $1$), we have
\begin{multline*}
	\int_{\R^{2d}}{\bm P}^\lambda(s,z;t,y)\, dz= \int_{\R^{2d}}\frac{\exp\left(-\frac{1}{2}(y+\zeta)^{\sf t}\mathcal{C}_\lambda^{-1}(t-s)(y+\zeta)\right)}{\sqrt{(2\pi)^{2d}{\sf det}(\mathcal{C}_\lambda(t-s))}}\, d\zeta\\
	=\int_{\R^{2d}}{\bm P}^\lambda(s,-e^{-(t-s)B}y;t,\zeta)d\zeta=1.
\end{multline*}

We also recall the following estimates on ${\bm P}^\mu$, that are given in \cite[Proposition~A1]{lucertinigood}.

\begin{prop}
For all $\delta, \lambda, T>0$ there exists a constant $C >0$ such that for any $0\le s < t \le T$ it holds:
\begin{equation}\label{eq:Gaussdercont}
	\left| \partial_{v_i} {\bm P}^\lambda(s,z;t,y)\right| \le \frac{C}{\sqrt{t-s}}{\bm P}^{\lambda+\delta}(s,z;t,y).
\end{equation}	
\end{prop}

In \cite[Theorem~1.1]{lucertinigood} the authors proved the existence of a fundamental solution for $\mathcal{L}$ under boundedness assumptions for $F$, together with some Gaussian bounds in terms of ${\bm P}^\lambda$ for suitable $\lambda>0$. 
Here, our first aim is to extend this result to our set of 
assumptions $(A_0)$, $(A_1)$ and $(A_2)$, i.e. with a {\it locally} $\alpha$-H\"older continuous $F$ having a {\it global} growth of order $\beta<\alpha$, 
specifically getting rid of the boundedness assumption imposed by \cite[Assumption 1.3]{lucertinigood}. 
Given our results are a refinement of the ones proposed in \cite{lucertinigood} obtained via a methodology similar to the one of
\cite{DeckKruse}, here we will only report the core computations required to achieve our results, while everything else will be left to the reader under the guidance of \cite{lucertinigood} and \cite{DeckKruse}.  

\begin{thm}
	\label{existence-fs2}
	Under Assumptions $(A_0)$, $(A_1)$ and $(A_2)$, the operator $\mathcal{L}$ has a fundamental solution $p(s,z; t, y)$, 
	defined for $s,t \in [0,+\infty)$ with $s<t$ and $z,y \in \R^{2d}$ in the sense of Definition \ref{def-fs2}. 
	Moreover, for any $\varepsilon,T > 0$ there exists a constant $C>0$ depending only on $T,\sigma,d,\varepsilon$ 
	and the constants $L,\alpha$ defined in Assumption $(A_2)$, such that 
	\begin{align}
		&p(s,z;t, y) \leq C \, {\bm P}^{\sigma + \varepsilon} (s,z;t, y) \label{gamstimL1}\\
		\label{gamstim} &|\partial_{v_i} p(s,z;t, y)| \leq \frac{C}{\sqrt{t - s}} {\bm P}^{\sigma + \varepsilon} (s,z;t, y)\\
		& |\partial^{2}_{v_iv_j} p(t,z;\mathcal{T}, y)| \leq \frac{C}{t - s} {\bm P}^{\sigma + \varepsilon} (s,z;t, y)
	\end{align}
	for any $t,s \in (0,T)$ with $s<t$, $y,z \in \R^{2d}$ and $i,j= 1, \ldots, d$. Furthermore, for any $T>0$ there exist two constants $\overline{\lambda},\overline{c}>0$ depending only on $T,\sigma,d,\varepsilon$ and the constants $L(T),\alpha(T)$ defined in Assumption $(A_2)$, such that 
	\begin{equation}\label{eq:loweGaus}
		{\overline c} \, {\bm P}^{\overline{\lambda}} (s,z; t, y) \le p(s,z; t, y) 
	\end{equation}
	for any $s,t \in (0,T)$ with $s<t$ and $y,z \in \R^{2d}$. Finally, for any $t,s \in [0,+\infty)$ with $t>s$ and $z \in \R^{2d}$, it holds
	\begin{equation}\label{eq:int1p}
		\int_{\R^{2d}}p(s,z;t,y)dy=1.
	\end{equation}
\end{thm}
\begin{rem}
	The integral identity \eqref{eq:int1p} is shown in \cite[Corollary 4.1]{lucertinigood}. Furthermore, the previous theorem has been proven in \cite{lucertinigood} for a fixed time horizon $T>0$, but it clearly holds under $(A_0)$, $(A_1)$  and $(A_2)$ by considering any time horizon $T>0$ and then observing that for $[0,T]\times \R^{2d}$ we are under the hypotheses of \cite[Theorem~1.1]{lucertinigood}. Hence, we know that for ay $T>0$ there exists a fundamental solution $p_T(s,z;t,y)$ for $s,t \in (0,T)$ with $s<t$ ad $y,z \in \R^{2d}$. The fact that in \cite[Theorem~1.1]{lucertinigood} the fundamental solution $p_T$ is constructed by means of the parametrix method guarantees that if we fix $T_0,T_1>0$ then it holds $p_{T_0}(s,z;t,y)=p_{T_1}(s,z;t,y)$ for any $s,t \in (0,\min\{T_0,T_1\})$ with $s<t$ and $y,z \in \R^{2d}$. Hence, the fundamental solution $p$ considered in Theorem \ref{existence-fs2} is defined as $p(s,z;t,y):=p_T(s,z;t,y)$ for any $T>0$, $t,s \in (0,T)$ with $s<t$ and $y,z \in \R^{2d}$.
\end{rem}
\begin{rem}\label{rem:nonconst}
	Clearly, the result holds true also when considering a more general diffusion governed by a symmetric matrix $(a_{i,j})_{i,j=1,\ldots, d}$ with a structural 
	uniformly ellipticity requirement and such that $a_{i,j} \in L^{\infty}([0, T]; C_b(\R^{2d})) \cap L^{\infty}([0,T]; C^{\alpha}_B(\R^{2d}))$ for every 
	$i,j =1, \ldots, d$, with also $\Psi \in L^{\infty}([0, T]; C_b(\R^{2d})) \cap L^{\infty}([0,T]; C^{\alpha}_B(\R^{2d}))$. Indeed, 
	the proof would follow as in \cite{lucertinigood} with the adaptation presented here for the lower order coefficient 
	$F$ extended also to the treatment of $\Psi$. Still, for the sake of readability, 
	we focus here only on the case of our interest. 
\end{rem}

\subsection{Proof of Theorem \ref{existence-fs2}}
	The proof of this result is a variation of both \cite[Theorem~1.1]{lucertinigood} and \cite[Theorem~2.1]{DeckKruse} and it is based on the Levy parametrix method, which presents many technicalities.
	For this reason, we will only present the proof of this result with unbounded coefficients for our specific case, but bear in mind it is possible to extend it to general Kolmogorov operators with unbounded coefficients. 
		
	Following the notation of \cite{lucertinigood}, we set our kernel ${\bm P}^{\sigma}(s,z; t, y)$ as the {\it parametrix}, whose explicit expression is the one 
	introduced in \eqref{eq:fun-sol-const}. Once the parametrix function is introduced, the Levy parametrix method prescribes that a fundamental solution to $\cL$ is of the form 
	\begin{align}
		\label{parametrix1}
		p(s,z; t, y) = {\bf P}^\sigma(s,z; t, y) + \int \limits_s^t  \int \limits_{\R^{2d}} {\bf P}^\sigma (s,z; \tau, \eta) \varphi(\tau, \eta; t,y) d \eta d \tau,
	\end{align}
	where $\varphi$ is an unknown function. Now, if we assume $p$ is indeed the fundamental solution to $\cL$, then we obtain
	\begin{align*}
		0 = \cL p(s,z ; t, y) = \cL {\bf P}^\sigma (s,z; t, y) +  \cL \int \limits_s^t  \int \limits_{\R^{2d}} {\bf P}^\sigma (s,z; \tau, \eta) \varphi(\tau, \eta; t,y) d \eta d \tau.
	\end{align*}
	Now we rewrite $\cL=\cK_\sigma+F(s,z)\cdot \nabla_v+Y$ and we recall that $(\mathcal{K}_\sigma+Y)\mathbf{P}^\sigma=0$. Hence $\varphi$ is the solution of the Volterra equation
	\begin{align*}
		\varphi(s,z;t,y)=F(s,z)\cdot\nabla_v {\bf P}^\sigma (s,z; t, y) +  \int \limits_s^t  \int \limits_{\R^{2d}} F(s,z)\cdot \nabla_v{\bf P}^\sigma (s,z; \tau, \eta) \varphi(\tau, \eta; t,y) d \eta d \tau,
	\end{align*}
	which can be determined through an iterative procedure that leads to the following Neumann series representation 
	\begin{align} \label{series-par}
		\varphi(s,z;t,y)= \sum \limits_{k \ge 1} \varphi_k(s,z; t,y) ,
	\end{align}
	where 
	\begin{equation}\label{eq:neumann}
		\begin{cases}
			\varphi_1(s,z; t,y) = F(s,z) \cdot \nabla_v {\bm P}^\sigma(s,z; t,y) , \\[20pt]
			\displaystyle \varphi_k(s,z; t,y) = \int \limits_s^t \int \limits_{\R^{2d}} 
			F(s,z) \cdot \nabla_v {\bm P}^\sigma(s,z; t,y)\, \varphi_k(\tau, \eta; t,y)  d \eta d\tau, \qquad k \in \mathbb{N}.
		\end{cases}
	\end{equation}
	Now, for the sake of brevity, we assume $\Psi \equiv 0$ bearing in mind that nothing in the proof would change if this was not the case, and we would have to proceed 
	exactly as in \cite{lucertinigood} to estimate the terms related to it. 
	In order to complete the proof of our result, we have to show that:
	\begin{enumerate} 
		\item the series \eqref{series-par} is uniformly convergent on compact subsets of $D_T := \{ (s,z;t,y) \, : \, 0 < s < t < T \}$ 
		(see forthcoming Subsection \ref{sub:conv}); moreover, upper bounds and H\"older estimates for $\varphi$ hold, 
		see following Subsection \ref{sub:gauss-par};
		\item $p$ defined in \eqref{parametrix1} is indeed a fundamental solution of $\cL$ and satisfies Gaussian estimates \eqref{gamstimL1}, see 
		Subsection \ref{sub:fun-sol}. 
	\end{enumerate}
	
	\subsubsection{\emph{(1)} Proof of the convergence of the Neumann series } 
	\label{sub:conv}
	In order to complete the proof of the uniform convergence of the series \eqref{series-par}, 
	one has to proceed as in \cite[Proposition 2.1]{lucertinigood} with some major changes directly linked to 
	the lack of boundedness for $F$, that we will explicitly exploit here. 
	In particular, we will not address at all how to deal with the term involving $\Psi$, since it 
	directly follows as in \cite[Proposition 2.1]{lucertinigood} (compare with the term $a$
	in their notation), but we will explicitly treat $F$.
	
	For a fixed $\varepsilon \in (0,1)$ and $\delta>0$ 
	for each $i=1, \ldots, d$ it holds 
	\begin{align} \label{parametrix-1}
		| F_i(s,z)\partial_{v_i} {\bf P}^\sigma(s,z; t,y)| 
		&\le \frac{C_T(F) C\left(1 + \| z\|_B^{\beta}\right){\bm P}^{\frac{\sigma+\delta}{\sqrt[d]{1-\varepsilon}}} (s,z; t,y)}{|t-s|^{\frac12}\sqrt{1-\varepsilon}}  \\
		&\quad \times \exp\left(-\left(-\frac{1}{2}(y-e^{(t-s)B}z)^{\sf t}\mathcal{C}_{\frac{\sigma+\delta}{\sqrt[d]{\varepsilon}}}^{-1}(t-s)(y-e^{(t-s)B}z)\right)\right)
	\end{align}
	where $C$ is the constant appearing in \eqref{eq:Gaussdercont} and $\beta$ is the parameter of assumption $(A_1)$.
	
	We now prove a preliminary Lemma concerning the product of the first two function on the right hand side,
	which can be seen as an extension of \cite[Lemma 3.5]{francesco2005class}.
	\begin{lem}\label{lem:genest}
		For any $z \in \R^{2d}$, $\sigma>0$ and $0 \le s<t$,
		\begin{align}\label{eq:genest}
			\| z \|_B^{\beta}\,  \exp\left(-\frac{1}{2}(y-e^{(t-s)B}z)^{\sf t} \mathcal{C}_{\sigma}^{-1}(t-s)(y-e^{(t-s)B}z)\right) \le 
			H_{\sigma,\beta}(y;t-s)
		\end{align}
		where, for $y=(\xi,\nu) \in \R^{2d}$ and $t>0$,
		\begin{equation*}
			H_{\sigma,\beta}(y;t)=C_\beta\left(|\nu| + \left( |\xi| + t|\nu|  \right)^{\frac13} + \left( (\beta \sigma^d + \beta^{\frac13}\sigma^{\frac{d}{3}})t \right)^{\frac12} \right)^{\beta}
		\end{equation*}
		and $C_\beta>0$ is a suitable constant depending only on $\beta$.
	\end{lem}
	\begin{proof}
		First, notice that
		\begin{equation*}
			\mathcal{C}_{\sigma}^{-1}(t-s)=\frac{1}{\sigma^d}\mathcal{C}_{1}^{-1}(t-s)
		\end{equation*}
		and
		\begin{multline}\label{eq:bellazio}
			\frac{1}{2}(y-e^{(t-s)B}z)^{\sf t} \mathcal{C}_{1}^{-1}(t-s)(y-e^{(t-s)B}z)\\
			= \frac{6}{(t-s)^3} \left((\xi - x) + \frac{t-s}{2} (\nu + v)\right)^2+\frac{1}{2(t-s)}(v-\nu)^2.
		\end{multline}
		Hence	
		\begin{align} \label{all-in-all-1}
			\| z \|_B^{\beta} &\exp\left(-\frac{1}{2}(y-e^{(t-s)B}z)^{\sf t} \mathcal{C}_{\sigma}^{-1}(t-s)(y-e^{(t-s)B}z)\right) \\
			&\le
			(| x |^{\frac{\beta}{3}} +  | v |^{\beta} )\exp\left(- \frac{6}{\sigma^d(t-s)^3} ((\xi - x) + \frac{t-s}{2} (\nu + v))^2\right)\\
			&\quad \times \exp\left(-\frac{1}{2\sigma^d (t-s)}(v-\nu)^2\right),	
		\end{align}
		where we recall $z=(x,v), y=(\xi, \nu) \in \R^{2d}$ and by $| \cdot |$ we denote the classic Euclidean norm in $\R^d$. Let us begin with the second term.  We have
		\begin{align} \nonumber
			|v|^{\beta} &\exp\left(- \frac{6}{\sigma^d(t-s)^3} ((\xi - x) + \frac{t-s}{2} (\nu + v))^2\right)
			\exp\left(-\frac{1}{2\sigma^d (t-s)}(v-\nu)^2\right)  \\
			&\qquad \le |v|^{\beta} \exp\left(-\frac{1}{2\sigma_1^d (t-s)}(v-\nu)^2\right) \\  \label{all-in-all-2}
			& \qquad \le \left( | \nu | + \left( \beta \sigma^d (t-s) \right)^{\frac12} \right)^{\beta},
		\end{align}
		where in the first step we considered that the exponential is bounded from above by $1$, and 
		in the second step we maximized the function as in the parabolic setting , see \cite[Equation (3.12)]{DeckKruse}.
		As for the first term, introducing a couple of auxiliary variables defined as 
		\begin{align*}
			w_z = x - \frac{t-s}{2} v , \qquad w_y = \xi + \frac{t-s}{2}\nu,
		\end{align*}	 
		we rewrite it as 
		\begin{align} \nonumber
			 \Big | w_z + v \frac{t-s}{2} \Big |^{\frac{\beta}{3}} &\exp\left(- \frac{6}{\sigma^d(t-s)^3} (w_z - w_y)^2\right)
			\exp\left(-\frac{1}{2\sigma^d (t-s)}(v-\nu)^2\right) \\ 
			\label{all-in-all-3}
			&\le C_\beta\left( |\xi| + (t-s)|\nu| + \left( \beta \sigma^d(t-s)^3\right)^{\frac12} \right)^{\frac{\beta}{3}} 
		\end{align}
		where we estimated each piece analogously as before and in the last but one line we substituted again the definition of $w_z$. Eventually, combining \eqref{all-in-all-1}, \eqref{all-in-all-2}, \eqref{all-in-all-3} we get \eqref{eq:genest}.
	\end{proof}

	Then \eqref{parametrix-1} becomes
	\begin{align*}
	&| F_i(s,z)\partial_{v_i} {\bf P}^\sigma(s,z; t,y)| \le  
		 \frac{C_T(F)C}{\sqrt{1-\varepsilon}}(1+H_{\frac{\sigma+\delta}{\sqrt[d]{\varepsilon}},\beta}(y;t-s))  		
		 \frac{{\bm P}^{\frac{\sigma+\delta}{\sqrt[d]{1-\varepsilon}}} (s,z; t,y)}{|t-s|^{\frac12}}.
	\end{align*}

	Hence, for any $\varepsilon \in (0,1)$,
	we obtain the following estimate for the first term of \eqref{series-par}:
	\begin{align}\label{eq:alleps}
		&|\varphi_1 (s,z; t,y)| \le  
		 \frac{dC_T(F)C}{\sqrt{1-\varepsilon}}(1+H_{\frac{\sigma+\delta}{\sqrt[d]{\varepsilon}},\beta}(y;T))  		
		 \frac{{\bm P}^{\frac{\sigma+\delta}{\sqrt[d]{1-\varepsilon}}} (s,z; t,y)}{|t-s|^{\frac12}},
	\end{align}
	where we recall that $C$ is independent of the choice of $\varepsilon$.
	Now, the idea is to proceed by induction and to estimate each term of \eqref{series-par} as previously done for $\varphi_1$. More precisely, we prove the following estimate.
	\begin{lem} \label{lemma:eq-induction}
		Let $(\varepsilon_j)_{j \ge 1}$ be a nonincreasing nonnegative sequence such that $S:=\sum_{j=1}^{+\infty}\varepsilon_j<1$. Then for any $n \ge 1$ and $(s,z;t,y) \in D_T$ it holds
		\begin{align} \label{eq:induction} 
				&|\varphi_n(s,z; t,y)|  \\ \nonumber
				&\qquad \le K^{n-1} \frac{\pi^{\frac{n}{2}}  (t-s)^{\frac{n-2}{2}} }{\Gamma(n/2)} \frac{d^nC_T^n(F)C^n}{(1-S)^{\frac{n}{2}} }  {\bm P}^{\frac{\sigma+\delta}{\sqrt[d]{1-\sum_{j=1}^n \varepsilon_j}}} (s,z; t,y)\prod_{j=1}^n  \left(1+H_{\frac{\sigma+\delta}{\sqrt[d]{\varepsilon_j}},\beta}(y;T)\right),
		\end{align}
		where $C_\beta$ is the constant given by Lemma \ref{lem:genest} and
		\begin{align}
				&K=\sqrt{1+\frac{\varepsilon_1}{1-S}}\left(1+J_\beta\left(T;\frac{\sigma+\delta}{\sqrt[d]{\varepsilon_1}} \right) +C_\beta (1+T^{\beta/3})\right), \label{eq:kay}\\	 
				&\qquad \text{with } J_\beta\left(T;\sigma\right)=C_\beta\left(T^{\frac{\beta}{3}}+\left((\beta \sigma^d+\beta^{\frac{1}{3}}\sigma^{\frac{d}{3}})T\right)^{\frac{\beta}{2}}\right). \label{eq:jay}
		\end{align}
	\end{lem}
	\begin{proof}
		We proceed by induction. First, we observe that by applying \eqref{eq:alleps} with $\varepsilon= \varepsilon_1 < S$,
		we easily get \eqref{eq:induction} for $n=1$. 
		
		In order to show the induction step, let us assume that \eqref{eq:induction} holds for some integer $n$ 
		and let us prove it holds for $n+1$.	
		Recalling the definition of $\varphi_n(s,z; t,y)$ as in \eqref{eq:neumann} we have
			\begin{align*}
			&|\varphi_{n+1} (s,z;t,y)| \\
			&\le \int \limits_s^t \int \limits_{\R^{2d}} 
			|F(s,z) \cdot \nabla_v   {\bm P} (s,z; \tau,\eta)|\, |\varphi_n(\tau, \eta; t,y) | d \eta d\tau \\
			&\le  K^{n-1} \frac{\pi^{\frac{n}{2}}}{\Gamma(n/2)} \frac{d^{n+1}C_T^{n+1}(F)C^{n+1}}{(1-S)^{\frac{n+1}{2}}}    
			\prod_{j=1}^n  \left(1+H_{\frac{\sigma+\delta}{\sqrt[d]{\varepsilon_j}},\beta}(y;T)\right)  \\
			&\quad\left(\int \limits_s^t  \frac{(t-\tau)^{\frac{n-2}{2}}}{|\tau-s|^{\frac12}}
			\int \limits_{\R^{2d}}
			{\bm P}^{\frac{\sigma+\delta}{\sqrt[d]{1-\sum_{j=1}^{n+1} \varepsilon_j}}} (s,z; \tau,\eta) 		
			{\bm P}^{\frac{\sigma+\delta}{\sqrt[d]{1-\sum_{j=1}^n \varepsilon_j}}} (\tau,\eta; t,y) d \eta d\tau\right.\\
			&\quad \quad +\left.\int \limits_s^t  \frac{(t-\tau)^{\frac{n-2}{2}}}{|\tau-s|^{\frac12}}
			\int \limits_{\R^{2d}}
			H_{\frac{\sigma + \delta}{\sqrt[d]{\sum_{j=1}^{n+1} \varepsilon_j} }} (\eta)
			{\bm P}^{\frac{\sigma+\delta}{\sqrt[d]{1-\sum_{j=1}^{n+1} \varepsilon_j}}} (s,z; \tau,\eta) 		
			{\bm P}^{\frac{\sigma+\delta}{\sqrt[d]{1-\sum_{j=1}^n \varepsilon_j}}} (\tau,\eta; t,y) d \eta d\tau\right).
		\end{align*}
		Now, we observe that for every $y \in \R^{2d}$ and $t >0$, it holds
		\begin{align}\label{eq:acca}
			H_{\sigma,\beta}(y;t) \le C_\beta\left((1+t^{\frac{\beta}{3}})\Norm{y}{B}^\beta+t^{\frac{\beta}{3}}+\left((\beta\sigma^d+\beta^{\frac{1}{3}}\sigma^{\frac{d}{3}})t\right)^{\frac{\beta}{2}}\right).
		\end{align}
		
		Hence, defining $J_\beta$ as in \eqref{eq:jay} and noticing it is an increasing function of $\sigma$, we get
		\begin{align}
			|\varphi_{n+1} (s,z;t,y)| 
			&\le  
			K^{n-1} \frac{\pi^{\frac{n}{2}}}{\Gamma(n/2)} \frac{d^{n+1}C_T^{n+1}(F)C^{n+1}}{(1-S)^{\frac{n+1}{2}}}    \left(\prod_{j=1}^n  \left(1+H_{\frac{\sigma+\delta}{\sqrt[d]{\varepsilon_j}},\beta}(y;T)\right)\right)  \\
			&\qquad \qquad 
			\left(\left( 1 + J_\beta \left( T,\frac{\sigma + \delta}{\sqrt[d]{\varepsilon_1 }} \right)\right)I_1+C_\beta \left( 1 + T^{\beta/3} \right)I_2\right). \label{eq:itut}
		\end{align}
		
		To handle the first integral $I_1$ , recall that
		\begin{multline}\label{eq:prectosucc}
		\frac{{\bm P}^{\frac{\sigma+\delta}{\sqrt[d]{1-\sum_{j=1}^n \varepsilon_j}}} (\tau,\eta; t,y)}{{\bm P}^{\frac{\sigma+\delta}{\sqrt[d]{1-\sum_{j=1}^{n+1} \varepsilon_j}}} (\tau,\eta; t,y)} = \sqrt{\frac{1-\sum_{j=1}^{n} \varepsilon_j}{1-\sum_{j=1}^{n+1} \varepsilon_j}}\\ \exp\left(-\frac{1}{2}(y-e^{(t-s)B}z)^{\sf t}\mathcal{C}_{\frac{\sigma+\delta}{\sqrt[d]{\varepsilon_{n+1}}}}^{-1}(t-s)(y-e^{(t-s)B}z)\right).
		\end{multline}
		By \eqref{eq:bellazio}, we know that
		\begin{equation*}
			{\bm P}^{\frac{\sigma+\delta}{\sqrt[d]{1-\sum_{j=1}^n \varepsilon_j}}} (\tau,\eta; t,y) \le \sqrt{\frac{1-\sum_{j=1}^{n} \varepsilon_j}{1-\sum_{j=1}^{n+1} \varepsilon_j}}{\bm P}^{\frac{\sigma+\delta}{\sqrt[d]{1-\sum_{j=1}^{n+1} \varepsilon_j}}} (\tau,\eta; t,y).
		\end{equation*}
		and then
		\begin{align}
			I_1 &:= \int \limits_s^t  \frac{(t-\tau)^{\frac{n-2}{2}}}{|\tau-s|^{\frac12}}
			\int \limits_{\R^{2d}}
			{\bm P}^{\frac{\sigma+\delta}{\sqrt[d]{1-\sum_{j=1}^{n+1} \varepsilon_j}}} (s,z; \tau,\eta) 		
			{\bm P}^{\frac{\sigma+\delta}{\sqrt[d]{1-\sum_{j=1}^n \varepsilon_j}}} (\tau,\eta; t,y) d \eta d\tau \\
			&\le \sqrt{\frac{1-\sum_{j=1}^{n} \varepsilon_j}{1-\sum_{j=1}^{n+1} \varepsilon_j}}{\bm P}^{\frac{\sigma+\delta}{\sqrt[d]{1-\sum_{j=1}^{n+1} \varepsilon_j}}} (s,z; t,y)\int \limits_s^t  \frac{(t-\tau)^{\frac{n-2}{2}}}{|\tau-s|^{\frac12}} d\tau \notag\\
			&\le \sqrt{\pi}\sqrt{1+\frac{\varepsilon_1}{1-S}} \frac{\Gamma\left(\frac{n}{2}\right)}{\Gamma\left(\frac{n+1}{2}\right)} (t-s)^{\frac{n-1}{2}} \label{eq:iuno}
		\end{align}
		where we applied the Chapman-Kolmogorov formula and we estimated as in \cite[Lemma 1.11]{AMP23}.
		Now, in order to control $I_2$ we observe that, by \eqref{eq:prectosucc} and Lemma~\ref{lem:genest}, we have
		\begin{equation*}
			\Norm{\eta}{B}^\beta {\bm P}^{\frac{\sigma+\delta}{\sqrt[d]{1-\sum_{j=1}^n \varepsilon_j}}} (\tau,\eta; t,y) \le  \sqrt{1+\frac{\varepsilon_{1}}{1-S}} H_{\frac{\sigma+\delta}{\sqrt[d]{\varepsilon_{n+1}}}, \beta}(y;T) {\bm P}^{\frac{\sigma+\delta}{\sqrt[d]{1-\sum_{j=1}^{n+1} \varepsilon_j}}} (\tau,\eta; t,y).
		\end{equation*}
		and then 
		\begin{align}
			I_2& := \int \limits_s^t  \frac{(t-\tau)^{\frac{n-2}{2}}}{|\tau-s|^{\frac12}}
			\int \limits_{\R^{2d}}
			\Norm{\eta}{B}^\beta
			{\bm P}^{\frac{\sigma+\delta}{\sqrt[d]{1-\sum_{j=1}^{n+1} \varepsilon_j}}} (s,z; \tau,\eta) 		
			{\bm P}^{\frac{\sigma+\delta}{\sqrt[d]{1-\sum_{j=1}^n \varepsilon_j}}} (\tau,\eta; t,y) d \eta d\tau  \\
			&\le \sqrt{1+\frac{\varepsilon_{1}}{1-S}} H_{\frac{\sigma+\delta}{\sqrt[d]{\varepsilon_{n+1}}}, \beta}(y;T){\bm P}^{\frac{\sigma+\delta}{\sqrt[d]{1-\sum_{j=1}^{n+1} \varepsilon_j}}} (s,z; t,y)\int \limits_s^t  \frac{(t-\tau)^{\frac{n-2}{2}}}{|\tau-s|^{\frac12}} d\tau \notag\\
			&\le \sqrt{\pi}\sqrt{1+\frac{\varepsilon_{1}}{1-S}} H_{\frac{\sigma+\delta}{\sqrt[d]{\varepsilon_{n+1}}}, \beta}(y;T) \frac{\Gamma\left(\frac{n}{2}\right)}{\Gamma\left(\frac{n+1}{2}\right)} (t-s)^{\frac{n-1}{2}}. \label{eq:idue}
		\end{align}
		
		Combining \eqref{eq:itut}, \eqref{eq:iuno} and \eqref{eq:idue} we finally obtain
		\begin{align*}
			|\varphi_{n+1} (s,z;t,y)| 
			&\le K^{n-1} \frac{\pi^{\frac{n+1}{2}}}{\Gamma\left(\frac{n+1}{2}\right)}\sqrt{1+\frac{\varepsilon_{1}}{1-S}} \frac{d^{n+1}C_T^{n+1}(F)C^{n+1}}{(1-S)^{\frac{n+1}{2}}}    \left(\prod_{j=1}^n  \left(1+H_{\frac{\sigma+\delta}{\sqrt[d]{\varepsilon_j}},\beta}(y;T)\right)\right)  \\
			&\qquad \left(\left( 1 + J_\beta \left( T,\frac{\sigma + \delta}{\sqrt[d]{\varepsilon_1 }} \right)\right)+C_\beta \left( 1 + T^{\beta/3} \right)H_{\frac{\sigma+\delta}{\sqrt[d]{\varepsilon_{n+1}}}, \beta}(y;T)\right)(t-s)^{\frac{n-1}{2}}\\
			&\le K^{n} \frac{\pi^{\frac{n+1}{2}} (t-s)^{\frac{n-1}{2}}}{\Gamma\left(\frac{n+1}{2}\right)} \frac{d^{n+1}C_T^{n+1}(F)C^{n+1}}{(1-S)^{\frac{n+1}{2}}}    \prod_{j=1}^{n+1}  \left(1+H_{\frac{\sigma+\delta}{\sqrt[d]{\varepsilon_j}},\beta}(y;T)\right).
		\end{align*}
		
	\end{proof}

	It remains to show that the Neumann series \eqref{series-par} is convergent.
	
	\begin{lem} \label{lemma:unif-conv}
		Set $\varepsilon_n=\varepsilon n^{\eta-\frac{1}{\beta}}$ with $0 < \eta < \frac{1}{\beta} - 1$. Then the Neumann series given by \eqref{series-par} is locally uniformly convergent in $D_T$.
	\end{lem}
	
	\begin{proof}
	Let us first observe that, since $H_{\sigma, \beta}(y;T)$ is an increasing function of $\sigma$, then by \eqref{eq:induction} we know that 	
	\begin{align}\label{eq:induction2}
		&|\varphi_n(s,z; t,y)|  \nonumber \\
		& \qquad \le \frac{(12)^{d/2}}{K(\sigma+\delta)^{d/2}(t-s)^{1+2d}\Gamma(n/2)}\left(\frac{KdC_T(F)C\sqrt{\pi(t-s)}}{\sqrt{1-S}}\left(1+H_{\frac{\sigma+\delta}{\sqrt[d]{\varepsilon_n}}}(y;T)\right)\right)^n.
	\end{align}
	Let us fix a compact set $E \subset D_T$. 
	Then there exists $M>0$ such that $0 < M^{-1} \le t-s \le T$ and $\max\{\Norm{z}{B},\Norm{y}{B}\}\le M$ for all $(s,z;t,y) \in E$. By \eqref{eq:acca} we know that
	\begin{equation*}
		H_{\frac{\sigma+\delta}{\sqrt[d]{\varepsilon_n}}}(y;T) \le C_\beta(1+T^{\beta/3})M^\beta+C_\beta T^{\beta/3}+J_{\beta}\left(T,\frac{\sigma+\delta}{\sqrt[d]{\varepsilon_n}}\right).
	\end{equation*}
	Setting
	\begin{equation*}
		M_1=\frac{K d C_T(F)C\sqrt{\pi T}(1+C_\beta(T^{\beta/3}+(1+T^{\beta/3})M^\beta))}{\sqrt{1-S}}, \ M_2=\frac{12^{d/2}M^{1+2d}}{K(\sigma+\delta)^{d/2}}
	\end{equation*}
	we get
	\begin{equation}\label{eq:induction3}
		|\varphi_n(s,z; t,y)| \le \frac{M_2M_1^n}{\Gamma(n/2)}\left(1+J_{\beta}\left(T,\frac{\sigma+\delta}{\sqrt[d]{\varepsilon_n}}\right)\right)^n.
	\end{equation}
	Recalling the asymptotics of the Euler Gamma function we know that there exists a constant $M_3 > 0$ such that $\Gamma (n/2) > M_3 \sqrt{n!} n^{-1/4} 2^{-n/2}$ and then
	\begin{equation}\label{eq:induction4}
		|\varphi_n(s,z; t,y)| \le \frac{M_2 M_1^n n^{1/4} 2^{n/2}}{M_3\sqrt{n!}}\left(1+J_{\beta}\left(T,\frac{\sigma+\delta}{\sqrt[d]{\varepsilon_n}}\right)\right)^n.
	\end{equation}
	Furthermore, for a given $M_4 >0$, by definition of $\varepsilon_n$ and \eqref{eq:jay}
	\begin{equation*}
		J_{\beta}\left(T,\frac{\sigma+\delta}{\sqrt[d]{\varepsilon_n}}\right)=C_\beta\left(T^{\beta/3}+\left(\left(\beta\frac{(\sigma+\delta)^d}{\varepsilon n^{\eta-\frac{1}{\beta}}}+\left(\beta\frac{(\sigma+\delta)^d}{\varepsilon n^{\eta-\frac{1}{\beta}}} \right)^{1/3}\right)T\right)^{\beta/2}\right) \le M_4\left(1+n^{-\frac{\eta\beta}{2}+\frac{1}{2}}\right)
	\end{equation*}
	and then, setting $M_5=M_1M_4$,
	\begin{equation}\label{eq:induction5}
		|\varphi_n(s,z; t,y)| \le \frac{M_2 M_5^n n^{1/4} 2^{n/2}}{M_3\sqrt{n!}}\left(1+n^{-\frac{\eta\beta}{2}+\frac{1}{2}}\right)^n.
	\end{equation}
	It remains to show that the series of the elements on the right-hand side of the estimate above converges. Indeed, we have
	\begin{align}\label{eq:induction4}
		&\sum_{n \geq 1}\frac{M_5^n n^{1/4} 2^{n/2}}{\sqrt{n!}}\left(1+n^{-\frac{\eta\beta}{2}+\frac{1}{2}}\right)^n \\ 
		&\le \sum_{n \geq 1} \frac{M_5^n n^{1/4} 2^{(n+2)/2}}{\sqrt{n!}}\left(1+n^{-\frac{n \eta\beta}{2}+\frac{n}{2}}\right) \\
		& \qquad = \sum_{n \geq 1} \frac{M_5^n n^{1/4} 2^{(n+2)/2}}{\sqrt{n!}}
		+ \sum_{n \geq 1} \frac{M_5^n n^{1/4} 2^{(n+2)/2}}{\sqrt{n!}}n^{-\frac{n \eta\beta}{2}+\frac{n}{2}} = S_1 + S_2 
	\end{align} 
	Let us begin with $S_1$. By applying the Cauchy-Schwartz inequality we get
	\begin{align*}
		S_1 = 2 \sum_{n \geq 1} \frac{(2M_5)^n n^{1/4}}{2^{n/2}\sqrt{n!}} 
		\leq 2 \sqrt{ \sum_{n \geq 1} \frac{(4M_5^2)^n}{n!} } \sqrt{ \sum_{n \geq 1}  \frac{ n^{1/2}}{2^{n}}  } < \infty,
	\end{align*}
	where the right-most product is convergent thanks to the quotient criteria.
	By analogously proceeding for $S_2$
	\begin{align*}
		S_2 &= 2 \sum_{n \geq 1} \frac{(2M_5)^n n^{1/4} }{2^{n/2} \sqrt{n!}}n^{-\frac{n \eta\beta}{2}+\frac{n}{2}} 
		\leq  2 \sqrt{ \sum_{n \geq 1} \frac{(4M_5^2)^n n^{1/2}e^n}{n! e^n}n^{-n \eta\beta+n}  }
			\sqrt{ \sum_{n \geq 1} \frac{ 1}{2^{n}}   }\\
		&\leq 2 C \sqrt{ \sum_{n \geq 1} (4M_5^2 e)^n n^{-n \eta\beta}  }
			\sqrt{ \sum_{n \geq 1} \frac{ 1}{2^{n}}   } \leq 2 C \sqrt{ \sum_{n \geq 1} (4M_5^2 e)^n n^{-n \eta\beta}  } < \infty,
	\end{align*}
	where it is possible to prove that the last product is convergent by Stirling's factorial asymptotics combined with the root criterion for series.
	\end{proof}
	\begin{rem}
		Observe that in Lemma~\ref{lemma:unif-conv} we only need $\beta<1$. However, if we consider $a_{i,j}$ as in Remark~\ref{rem:nonconst}, the estimate \eqref{eq:induction} becomes
		\begin{equation*}
			|\varphi_n(s,z;t,y)| \le \overline{K}^{n-1}\frac{\left(\Gamma\left(\frac{\alpha}{2}\right)\right)^{\frac{n}{2}}(t-s)^{\frac{\alpha n-2}{2}}}{\Gamma\left(\frac{n\alpha}{2}\right)}{\bm P}^{\frac{\sigma+\delta}{\sqrt[d]{1-\sum_{j=1}^n \varepsilon_j}}} (s,z; t,y)\prod_{j=1}^n  \left(1+H_{\frac{\sigma+\delta}{\sqrt[d]{\varepsilon_j}},\beta}(y;T)\right),
		\end{equation*}
		for some constant $\overline{K}>0$. Using this formula in Lemma~\ref{lemma:unif-conv}, the asymptotics of the Gamma function imply the convergence of the Neumann series under $\beta<\alpha$. 
	\end{rem}

	\subsubsection{\emph{(1)} Proof of Gaussian estimates and H\"older continuity for $\varphi$}
	\label{sub:gauss-par}
	The proof of these results follows as in \cite[Lemma 6.1]{francesco2005class} with
	suitable modifications to account for our choice of the parametrix. In particular, in \cite[Lemma 4.3]{francesco2005class} one needs to apply 
	Lemma \ref{lem:genest}
	iteratively in the same spirit of the proof of Lemma \ref{lemma:eq-induction}. Here, the condition $\beta<\alpha$ is needed even in the constant $\sigma$ case.
		
	\subsubsection{\emph{(2)} $p$ is a fundamental solution of $\cL$ and satisfies the Gaussian estimates \eqref{gamstimL1} } 
	\label{sub:fun-sol}
	As far as we are concerned with point (2) of the proof, we need to show that our candidate function $p$ introduced in \eqref{parametrix1} is
	indeed a fundamental solution in the sense of Definition \ref{def-fs2}. This can be carried out as in \cite{lucertinigood}.

\subsection{Useful applications of Theorem \ref{existence-fs2}}
Next, we also extend the existence and uniqueness result for the Cauchy problem proved in \cite[Theorem 4.1]{lucertinigood} to our setting.
\begin{thm}\label{thm:exstuniqueadj}
	Let Assumptions $(A_0)$, $(A_1)$ and $(A_2)$ hold and fix $T>0$. Assume further that:
	\begin{itemize}
		\item[$(i)$] there exists $\gamma_{\Psi} \in [0,1)$ and $\beta_\Psi \in (0,\alpha)$ such that
		\begin{equation*}
			\Norm{\Psi}{L^\infty_{\gamma}((0,T);C_{\rm B}( \R^{2d}))}:=\esssup_{t \in (0,T)}(T-t)^{\gamma}\Norm{\Psi(t)}{C_{\rm B}^{\beta_\Psi}}<\infty.
		\end{equation*}
		\item[$(ii)$] $g \in C_{\rm B}^{\beta_g}(\R^{2d})$ for some $\beta_g \in [0,2+\beta_\Psi]$ (where $C_{\rm B}^0(\R^{2d})=C_{\rm b}(\R^{2d})$).
	\end{itemize}
	Then the function
	\begin{equation}\label{eq:soladj}
		u(t,z)=\int_{\R^{2d}}p(t,z;T,y)g(y)dy-\int_t^T\int_{\R^{2d}}p(t,z;\tau,y)\Psi(\tau,y)\, dy \, d\tau
	\end{equation}
	is the unique bounded strong Lie solution of \eqref{eq:adjeq-app}.
	Furthermore, for any $t \in (0,T)$, $u \in C^{2+\beta_\Psi}_{{\rm B},Y}((0,t)\times \R^{2d})$ and there exists a constant $C_u>0$ independent of $t \in (0,T)$ such that
	\begin{equation}\label{eq:Holderbound}
		\Norm{u}{C^{2+\beta_\Psi}_{{\rm B},Y}((0,t)\times \R^{2d})} \le C_u\left((T-t)^{-\frac{2+\beta_\Psi-\beta_g}{2}}\Norm{g}{C_{\rm B}^{\beta_g}(\R^{2d})}+(T-t)^{-\gamma}\Norm{\Psi}{L^\infty_{\gamma}((0,T);C_{\rm B}^{\beta_\Psi}(\R^{2d}))}\right).
\end{equation}
\end{thm}
\begin{rem}
	The formula \eqref{eq:soladj} perfectly explains why $p$ is called \textit{fundamental solution} of $\cL^\ast$, as it is the \textit{fundamental} building block of the bounded solutions of \eqref{eq:adjeq-app}. For analogy, we could also refer to $p$ as the \textit{heat kernel} of $\cL^\ast$.
	
	Furthermore, \eqref{eq:soladj} can be also used under the conditions:
	\begin{itemize}
		\item[$(iii)$] There exist $\gamma \in [0,2]$ and $C>0$ such that
		\begin{equation*}
			|\Psi(t,z)|+|g(z)| \le Ce^{C|z|^\gamma}, \ (t,z) \in (0,T)\times \R^{2d}
		\end{equation*}
		\item[$(iv)$] There exists $\beta \in (0,1]$ such that for any compact $K \subset \R^{2d}$
		\begin{equation*}
			\sup_{\substack{t \in (0,T), \, z_1,z_2 \in K \\ z_1 \not = z_2}}\frac{|\Psi(t,z_1)-\Psi(t,z_2)|}{\Norm{z_1-z_2}{\rm B}^{\beta}}<\infty. 
		\end{equation*}
	\end{itemize}
	Indeed, if hypotheses $(iii)$ and $(iv)$ hold in place of $(i)$ and $(ii)$ in Theorem \ref{thm:exstuniqueadj}, then there exists $T_0$ such that for any $T<T_0$ the function $u$ in \eqref{eq:soladj} is the unique strong Lie solution of \eqref{eq:adjeq-app} that satisfies
	\begin{equation*}
		|u(t,z)|\le C_ue^{C_u|z|^2}, \ (t,z) \in (0,T)\times \R^{2d}.
	\end{equation*}
	The case $\gamma=2$ has been stated in \cite[Theorem~4.1 and Remark~4.1]{lucertinigood}, and it is obtained by means of standard methods in the theory of parabolic partial differential equations (as for instance done in \cite{francesco2005class}, see also the book \cite{pascucci2011pde}). Analogously, if $\gamma<2$, then with the same arguments one obtains $T_0=\infty$. In such a way, one has a Tychonoff-like uniqueness condition.
\end{rem}

We will consider, in particular, the case $g \equiv 0$ and $\Psi(t,z)=\psi(z)$ for any $t>0$ and $z \in \R^{2d}$, where $\psi \in C^{\infty}_c(\R^{2d})$. 
It is not difficult to prove that hypothesis $(i)$ of Theorem~\ref{thm:exstuniqueadj} is satisfied by our choice of $\Psi$, with $\gamma=0$. 
Hypothesis $(ii)$ is clearly satisfied for $\beta_g=2+\beta_\Psi$. Hence, we can state the following corollary, which is a direct consequence of Theorem \ref{thm:exstuniqueadj}.
\begin{cor} \label{corollary-app}
	Let Assumption $(A_0), (A_1)$ and $(A_2)$ hold true. Let also $\psi \in C^\infty_{\rm c}(\R^{2d})$. Then for any $T>0$ the backward Kolmogorov equation
	\begin{equation}\label{eq:adjeqcinf}
		\begin{cases}
			\cL u(t,z)=\psi(z) &(t,z) \in [0,T) \times \R^{2d} \\
			u(T,z)=0 & z \in \R^{2d}
		\end{cases}
	\end{equation}
	admits a unique bounded strong Lie solution
	\begin{equation}\label{eq:adjsolcinf}
		u(t,z)=-\int_t^T\int_{\R^{2d}}p(t,z;\tau,y)\psi(y)\, dy \, d\tau.
	\end{equation}
	Furthermore, for any $\beta \in (0,\alpha(T))$ there exists a constant $C_u>0$
	\begin{equation*}
		\Norm{u}{C_{{\rm B},Y}^{2+\beta}((0,T)\times \R^{2d})} \le C_uC_\beta(\psi),
	\end{equation*}
	where
	\begin{equation}\label{constantpsi}
		C_{\beta}(\psi):=\max\{2^{1-\beta}\Norm{\psi}{L^\infty(\R^{2d})}\Norm{\nabla_z \psi}{L^\infty(\R^{2d})}^{\beta},\Norm{\nabla_z\psi}{L^\infty(\R^{2d})}\}.	
	\end{equation}
	As a consequence
	\begin{equation}\label{eq:Linfbound}
		\Norm{u}{L^\infty((0,T)\times \R^{2d})}+\Norm{\nabla_v u}{L^\infty((0,T)\times \R^{2d})} \le C_uC_\beta(\psi).
	\end{equation}
\end{cor}
\begin{rem}
	Actually, by \eqref{eq:adjsolcinf}, we know that
	\begin{equation*}
		\Norm{u}{L^\infty((0,T)\times \R^{2d})} \le T\Norm{\psi}{L^\infty(\R^{2d})}
	\end{equation*}
	where we used the fact that the fundamental solution is nonnegative (by \eqref{eq:loweGaus}) and \eqref{eq:int1p}. Furthermore, by \eqref{gamstimL1} and \eqref{eq:integ1Gam2}, we know that there exists a constant $C>0$ such that
	\begin{equation}\label{eq:L1bound}
		\Norm{u}{L^1((0,T)\times \R^{2d})} \le C\Norm{\psi}{L^1(\R^{2d})}.
	\end{equation}
\end{rem}

\medskip
\section{Proof of Theorem \ref{thm:existence}}
\label{proofthm}
We first prove uniqueness of the solutions and then we move to the existence statement. To do this, we need to better understand the relation between the solutions of \eqref{eq:adjeqcinf} and the ones of \eqref{probmu} in $[0,T]$. This is explained in the following lemma.
\begin{lem}\label{lem:quellocheeraunaclaim}
	Let the assumptions of Theorem \ref{thm:existence} hold. For any $\psi \in C^\infty_{\rm c}(\R^{2d})$, let $u:[0,T] \times \R^{2d} \to \R$ be the unique bounded strong Lie solution of \eqref{eq:adjeqcinf}, that exists by Corollary \ref{corollary-app}, and let $\mu$ be any solution of \eqref{probmu}. Then
	\begin{equation*}
		t \in [0,T] \mapsto \int_{\R^{2d}}u(t,z)d\mu_t(z)
	\end{equation*}
	is Lipschitz continuous and for any $t \in [0,T]$ it holds
	\begin{equation}\label{eq:repr11}
		\int_{\R^{2d}}u(t,z)d\mu_t(z)=-\int_t^T \int_{\R^{2d}} \psi(z)d\mu_s(z) ds.
	\end{equation}	
\end{lem}
\begin{proof}
	Let $u$ be the bounded strong Lie solution of \eqref{eq:adjeqcinf} and extend it by setting \linebreak $u(t,z)=0$ for any $t>T$ and $z \in \R^{2d}$. Fix $\varepsilon>0$ let $u_\varepsilon=\rho_\varepsilon \, \ast \, u$.  Notice that \linebreak $u_\varepsilon \in C^\infty\left(\left(\sfrac{5}{2}\, \varepsilon^2,+\infty\right) \times \R^d\right)$ and that, by \eqref{eq:L1bound} and Lemma \ref{conv-dert}, 
	\begin{multline}
		\Norm{\partial_t u_\varepsilon}{L^\infty\left(\left(\sfrac{5}{2}\,\varepsilon^2,\infty\right) \times \R^{2d}\right)}+\Norm{\nabla_x u_\varepsilon}{L^\infty\left(\left(\sfrac{5}{2}\,\varepsilon^2,\infty\right) \times \R^{2d}\right)}\\+\Norm{\Delta_vu_\varepsilon}{L^\infty\left(\left(\sfrac{5}{2}\,\varepsilon^2,\infty\right) \times \R^{2d}\right)} \le C\Norm{\psi}{L^1(\R^{2d})}. \label{eq:boundonder}
	\end{multline}
	Furthermore, by \eqref{eq:Linfbound} and the fact that the convolution commutes with $\nabla_v$, we get
	\begin{equation}\label{eq:moltoreg}
		\Norm{\nabla_v u_\varepsilon}{L^\infty\left(\left(\sfrac{5}{2}\,\varepsilon^2,\infty\right) \times \R^{2d}\right)} \le C_u C_\beta(\psi)
	\end{equation}
	for any fixed $\beta \in (0,\alpha)$. We underline that this bound is independent of $\varepsilon>0$. In particular, this means that for fixed $s \in \left(\sfrac{5}{2}\, \varepsilon^2, T\right]$, $u_\varepsilon(s,\cdot) \in {\sf A}(\R^{2d})$ with bounds that are independent of $s$. 
	
	Let us now prove that, for fixed $\varepsilon>0$,
	\begin{equation*}
		t \in \left(\sfrac{5}{2}\, \varepsilon^2, T\right] \mapsto \int_{\R^{2d}}u_\varepsilon (t,z)d\mu_t
	\end{equation*}
	is Lipschitz-continuous, with Lipschitz constant depending on $\varepsilon>0$. 
	
	In order to do this, we rewrite the difference as
	\begin{align}\label{eq:passpreuni2}
		\begin{split}
			\int_{\R^{2d}} &u_\varepsilon(t,z)d\mu_t -\int_{\R^{2d}}u_\varepsilon(s,z)d\mu_s 
			\\
			&=\int_{\R^{2d}} (u_\varepsilon(t,z)-u_\varepsilon(s,z))d\mu_t +\int_{\R^{2d}}u_\varepsilon(s,z)d(\mu_t-\mu_s) \\
			&=  \int_{s}^{t}\int_{\R^{2d}} \partial_t u_\varepsilon(h,z) d\mu_t(z)\, dh 
			+ \int_{\R^{2d}}u_\varepsilon(s,z)d(\mu_s-\mu_t)(z),			
		\end{split}
	\end{align}
	where the first integral on the right-hand side is well-posed by \eqref{eq:boundonder}. Furthermore, since $u_\varepsilon(s,\cdot) \in {\sf A}(\R^{2d})$ we can use \eqref{eq:weaksol} to rewrite the second integral on the right-hand side of \eqref{eq:passpreuni2} as follows 
	\begin{align}
		&\int_{\R^{2d}} u_\varepsilon(t,z)d\mu_t -\int_{\R^{2d}}u_\varepsilon(s,z)d\mu_s 
		=  \int_{s}^{t}\int_{\R^{2d}} \partial_t u_\varepsilon(h,z)d\mu_t(z) \, dh  \label{eq:equalityueps2}\\
		&\quad \quad 
		+ \int_s^t\int_{\R^{2d}} \left(v \cdot \nabla_x u_\varepsilon(s,z)+F(t,z)\cdot \nabla_v u_\varepsilon(s,z)
		+ \frac{\sigma^2}{2}\Delta_v u_\varepsilon(s,z)\right)\, d\mu_h(z) \, dh. \notag
	\end{align}
	Taking the absolute value on both sides and using the aforementioned bounds, we get
	\begin{align*}
		&\left|\int_{\R^{2d}}u_\varepsilon(t,z)d\mu_t -\int_{\R^{2d}}u_\varepsilon(s,z)d\mu_s\right| \\
		&\qquad \le  \int_{s}^{t}\int_{\R^{2d}} |\partial_t u_\varepsilon(h,z)|d\mu_t(z) \, dh \\
		&\qquad \qquad 
		+ \int_s^t\int_{\R^{2d}} \left|v \cdot \nabla_x u_\varepsilon(s,z)+F(t,z)\cdot \nabla_v u_\varepsilon(s,z)
		+\frac{\sigma^2}{2}\Delta_v u_\varepsilon(s,z)\right|d\mu_h(z) \, dh\\
		&\qquad \le C \left(|t-s|+\int_s^t\int_{\R^{2d}}|v|d\mu_h(z)dh\right)\le C|t-s|
	\end{align*}
	where we used that $\{\mu_t\}_{t \in [0,T]} \in C([0,T];\W_1(\R^{2d}))$ and thus its first moments are uniformly bounded. Furthermore, using again \eqref{eq:equalityueps2}, it is clear that
	\begin{align*}
		\frac{d}{d \, t}\int_{\R^{2d}}&u_\varepsilon(t,z)d\mu_t(z)\\
		&=\int_{\R^{2d}}\left(\partial_t u_\varepsilon(t,z)+v \cdot \nabla_x u_\varepsilon(t,z)+F(t,z)\cdot \nabla_v u_\varepsilon(t,z)+\frac{\sigma^2}{2}\Delta_v u_\varepsilon(t,z)\right)d\mu_t(z)\\
		&=\int_{\R^{2d}}\left(Yu_\varepsilon(t,z)+F(t,z)\cdot \nabla_v u_\varepsilon(t,z)+\frac{\sigma^2}{2}\Delta_v u_\varepsilon(t,z)\right)d\mu_t(z).
	\end{align*}
	
	Then, integrating over $[t,T]$ for $t>\sfrac{5}{2}\, \varepsilon^2$, we get
	\begin{align*}
		\int_{\R^{2d}}&u_\varepsilon(T,z)\, d\mu_T(z)-\int_{\R^{2d}}u_\varepsilon(t,z)\, d\mu_t(z)\\
		&=\int_t^T\int_{\R^{2d}}\left(Yu_\varepsilon(s,z)+F(s,z)\cdot \nabla_v u_\varepsilon(s,z)+\frac{\sigma^2}{2}\Delta_v u_\varepsilon(s,z)\right)d\mu_s(z) \, ds.
	\end{align*}
	First of all, we observe that as $\varepsilon \to 0$ the left-hand side converges to
	\begin{equation*}
		\int_{\R^{2d}}u(T,z)\, d\mu_T(z)-\int_{\R^{2d}}u(t,z)\, d\mu_t(z)=-\int_{\R^{2d}}u(t,z)\, d\mu_t(z)
	\end{equation*}
	by the dominated convergence theorem, since by definition and \eqref{eq:Linfbound}
	\begin{equation*}
		\Norm{u_\varepsilon}{L^\infty\left(\left(\sfrac{5}{2}\, \varepsilon^2,+\infty\right)\times \R^{2d}\right)} \le \Norm{u}{L^\infty((0,+\infty)\times \R^{2d})} \le C_uC_\beta(\psi).
	\end{equation*}
	For the right hand side, we rewrite
	\begin{align*}
		\int_t^T&\int_{\R^{2d}}\left(Yu_\varepsilon(s,z)+F(s,z)\cdot \nabla_v u_\varepsilon(s,z)+\frac{\sigma^2}{2}\Delta_v u_\varepsilon(s,z)\right)d\mu_s(z)\\
		&=\int_t^T\int_{\R^{2d}}\left(Yu_\varepsilon(s,z)+(\rho_\varepsilon \ast (F\cdot \nabla_v u))(s,z)+\frac{\sigma^2}{2}\Delta_v u_\varepsilon(s,z)\right)d\mu_s(z)\, ds \\
		&\quad + \int_t^T\int_{\R^{2d}}\left(F(s,z)\cdot \nabla_v u_\varepsilon(s,z)-(\rho_\varepsilon \ast (F\cdot \nabla_v u))(s,z)\right)d\mu_s(z) \, ds\\
		&=\int_t^T\int_{\R^{2d}}\left(\rho_\varepsilon \ast \left(Yu+F\cdot \nabla_v u+\frac{\sigma^2}{2}\Delta_v u\right)\right)(s,z)d\mu_s(z)\, ds \\
		&\quad + \int_t^T\int_{\R^{2d}}\left(F(s,z)\cdot \nabla_v u_\varepsilon(s,z)-(\rho_\varepsilon \ast (F\cdot \nabla_v u))(s,z)\right)d\mu_s(z) \, ds\\
		&=\int_t^T\int_{\R^{2d}}(\rho_\varepsilon \ast \psi)(s,z)d\mu_s(z)\, ds \\
		&\quad + \int_t^T\int_{\R^{2d}}F(s,z)\cdot \nabla_v u_\varepsilon(s,z)\, d\mu_s(z)\, ds-\int_t^T\int_{\R^{2d}}\rho_\varepsilon \ast (F\cdot \nabla_v u)(s,z)\, d\mu_s(z) \, ds,
	\end{align*}
	where in the second last equality we used Lemma \ref{lemY} while in the last one we used Remark \ref{rem:pointwise2}. It is clear that, by the dominated convergence theorem,
	\begin{equation*}
		\lim_{\varepsilon \to 0}\int_t^T\int_{\R^{2d}}(\rho_\varepsilon \ast \psi)(s,z)d\mu_s(z)\, ds=\int_t^T\int_{\R^{2d}} \psi(z)d\mu_s(z)\, ds,
	\end{equation*}
	hence we only need to prove that the second and third integral converge to the same value. To do this, one has simply to apply Lemma \ref{lem:L1Cara} with $m=d$ to the bounded Caratheodory function $\nabla_v u$ and the sublinear one $F$, by $(A_2)$, in the second integral, and with $m=1$ to the bounded Caratheodory function $1$ and the sublinear one  $F \cdot \nabla_v u$ in the third one. Notice that we need estimate  \eqref{eq:Linfbound} to guarantee that $\nabla_v u$ is bounded (and $F \cdot \nabla_v u$ is sublinear by $(A_2)$) and then the lemma can be applied. The Lipschitz property and the fact that \eqref{eq:repr11} holds up to $t=0$ follow immediately.  	
\end{proof}
Now we are ready to prove the following uniqueness result.
\begin{prop}\label{prop:uniqueness}
	Under the assumptions of Theorem~\ref{thm:existence}, Equation~\eqref{probmu} admits at most one solution.
\end{prop}
\begin{proof}
	Let $\bm{\mu}^1,\bm{\mu}^2\in C(\R_0^+; \W_1(\R^{2d}))$ be two solutions of \eqref{probmu} and consider $\psi \in C^\infty_c(\R^{2d})$. By Lemma \ref{lem:quellocheeraunaclaim}, for any $\psi \in C^\infty_c(\R^{2d})$, we have 
	\begin{equation*}
		\int_{\R^{2d}}u(t,z)d\mu^j_t(z)=-\int_t^T\int_{\R^{2d}}\psi(z)d\mu^j_s(z)ds, \ j=1,2,
	\end{equation*}
	where $u$ is the unique bounded strong Lie solution of \eqref{eq:adjeqcinf}. Subtracting term by term and evaluating at $0$, we get
	\begin{equation*}
		\int_0^T\int_{\R^{2d}}\psi(z)d(\mu^1_s-\mu^2_s)(z)ds=0,
	\end{equation*}
	whence the thesis by arbitrariness of $\psi$. This clearly applies also to global solutions, by arguing with an arbitrary $T>0$.	
\end{proof}









To prove the existence of a solution, we need, instead, to relate the strong solutions of \eqref{eq:SDEstrongaux-intro} with the ones of \eqref{probmu}. This is established in the following proposition.
\begin{prop}\label{prop:Ito}
	Let the assumptions of Theorem \ref{thm:existence} hold and let $(X,V) \in L^1(\Omega;C([0,T];\R^{2d}))$ be a strong solution of \eqref{eq:SDEstrongaux-intro}. Then the flow of probability measures $\bm{\mu} \in \W_1(C([0,T];\R^{2d}))$, where $\bm{\mu}={\rm Law}(X,V)$, solves \eqref{probmu}.
\end{prop}
\begin{proof}
	 Let $Z=(X,V)$ be a strong solution of \eqref{eq:SDEstrongaux-intro} and $\psi \in C_c^{\infty}(\R^{2d})$. By It\^o's formula (see \cite[Theorem IV.3.3]{revuzyor}) we have
	 \begin{align}\label{eq:Ito}
	 	\begin{split}
	 		\psi(Z(t))&=\psi(Z(0))\\
	 		&+\int_0^t \left(V(s)\cdot \nabla_x \psi(Z(s))+F(s,Z(s))\cdot \nabla_v \psi(Z(s))+\sigma\Delta_v \psi(Z(s))\right)ds\\
	 		&+\sqrt{2\sigma} \int_0^t\nabla_v\psi(Z(s))\, dB(s)		
	 	\end{split}
	 \end{align}
	 Since $\psi \in C^\infty_c(\R^{2d})$, it is clear that $\nabla_v \psi$ is bounded and then (see \cite[Theorem 3.2.1]{oksendal2013stochastic})
	 \begin{equation}\label{eq:zeromean}
	 	\E\left[\int_0^t\nabla_v\psi(Z(s))\, dB(s)\right]=0.
	 \end{equation}
	 Furthermore, observe that, given $\psi \in C^\infty_c(\R^{2d})$ and Assumption $(A_1)$,
	 \begin{equation*}
	 	\left|V(s)\cdot \nabla_x \psi(Z(s))+F(s,Z(s))\cdot \nabla_v \psi(Z(s))+\sigma\Delta_v \psi(Z(s))\right| \le C(1+\sup_{0\le s \le T}|Z(s)|),
	 \end{equation*}
	 where $C$ is a constant depending on $\psi$, $C(F)$ and $\sigma$. Since it is clear that
	 \begin{multline*}
	 	\E\left[\int_0^t \left|V(s)\cdot \nabla_x \psi(Z(s))+F(s,Z(s))\cdot \nabla_v \psi(Z(s))+\sigma\Delta_v \psi(Z(s))\right|\, ds\right] \\
	 	\le CT\left(1+\E\left[\sup_{0 \le s \le T}|Z(s)|\right]\right)<\infty,
	 \end{multline*}
	we can use Fubini's theorem to get
	 \begin{multline}\label{eq:changemean}
	 	\E\left[\int_0^t \left(V(s)\cdot \nabla_x \psi(Z(s))+F(s,Z(s))\cdot \nabla_v \psi(Z(s))+\frac{\sigma^2}{2}\Delta_v \psi(Z(s))\right)\, ds\right]\\
	 	=\int_0^t \int_{\R^{2d}}\left(v\cdot \nabla_x \psi(z)+F(s,z)\cdot \nabla_v \psi(z)+\frac{\sigma^2}{2}\Delta_v \psi(z)\right)\, d\mu_s(z) \, ds.
	 \end{multline}
	 Finally, taking the expectation on both sides of \eqref{eq:Ito} and using \eqref{eq:zeromean} and \eqref{eq:changemean}, we get \eqref{eq:weaksol}. 
\end{proof}
Now we are finally ready to prove the main result.
\begin{proof}[Proof of Theorem \ref{thm:existence}]
Let $(X_0,V_0) \in \cM(\Omega;\R^{2d})$ be independent of $B$ and such that ${\rm Law}(X_0,V_0)=\mu_0$. To prove that \eqref{eq:SDEstrongaux-intro} admits a strong solution, we need to consider two cases:

\textbf{Case 1: $F$ is bounded.} Assume that $F \in L^{\infty}([0, + \infty) \times \R^{2d})$.
Fix $T>0$ and let $\cT:\Omega \times C([0,T];\R^{2d}) \to C([0,T];\R^{2d})$ be defined as follows: for any $\omega \in \Omega$ and $(X,V) \in C([0,T];\R^{2d})$, we set 
\begin{equation*}
\cT(\omega,X,V)=(\cT_X(\omega,X,V),\cT_V(\omega,X,V))	
\end{equation*}
where for any $t \in [0,T]$,
\begin{align*}
	\cT_X(\omega,X,V)(t)&=X_0(\omega)+\int_0^t\cT_V(\omega,X,V)(s)\, ds\\
	\cT_V(\omega,X,V)(t)&=V_0(\omega)+\int_0^tF(s,X(s),V(s))\, ds+\sqrt{2\sigma} B(\omega,t).
\end{align*}
For fixed $(X,V) \in C([0,T];\R^{2d})$, $\cT_V(\cdot,X,V)$ is sum of measurable functions of $\omega \in \Omega$ while and $\cT_X(\cdot,X,V)$ is the sum of a measurable function of $\omega$ with the Riemann integral of a continuous stochastic process, hence $\cT$ is a random operator. Now we want to prove that $\cT$ is a continuous and compact random operator.


We first prove that $\cT$ is a continuous random operator. To do this, fix $\omega \in \Omega$ and consider any sequence $(X_n,V_n) \in C([0,T];\R^{2d})$ converging towards $(X,V) \in C([0,T];\R^{2d})$. Then we have for any $t \in [0,T]$
\begin{equation*}
	\left|\cT_V(\omega,X,V)(t)-\cT_V(\omega,X_n,V_n)(t)\right| \le \int_0^t |F(s,X(s),V(s))-F(s,X_n(s),V_n(s))|\, ds.
\end{equation*}
Taking the supremum over $[0,T]$, we get
\begin{equation*}
	\Norm{\cT_V(\omega,X,V)-\cT_V(\omega,X_n,V_n)}{L^\infty[0,T]} \le \int_0^T|F(s,X(s),V(s))-F(s,X_n(s),V_n(s))|\, ds.
\end{equation*}
Recalling that $F$ is bounded and continuous in the second variable by Assumption $(A_0)$, we can use the dominated convergence theorem to achieve
\begin{equation*}
	\lim_{n \to +\infty}\Norm{\cT_V(\omega,X,V)-\cT_V(\omega,X_n,V_n)}{L^\infty[0,T]}=0.		
\end{equation*}
Once this is done, it is also clear that $\cT_X(\omega,X_n,V_n)$ converges towards $\cT_X(\omega,X,V)$. This proves the continuity of $\cT(\omega, \cdot, \cdot)$ for any $\omega \in \Omega$.

Next, we need to prove that $\cT$ is a compact random operator. To do this we will show that for fixed $\omega \in \Omega$ the range of $\cT(\omega,\cdot,\cdot)$ is composed of equibounded and equi-H\"older functions. Indeed, in such a case, by the Ascoli-Arzel\'a theorem we know that $\cT$ is compact random operator.

So, first, 
we prove that for any $\omega \in \Omega$ there exists a constant $M(\omega)$ such that 
\begin{equation}\label{eq:step2}
	\Norm{\cT(\omega,X,V)}{L^\infty[0,T]} \le M(\omega), \ \quad \forall (X,V) \in C([0,T];\R^{2d}).
\end{equation}
Indeed, we set 
\begin{equation}\label{eq:MB}
	M_B(\omega):=\max_{t \in [0,T]}|B(\omega,t)|	
\end{equation}
and observing that for any $t \in [0,T]$ and any $(X,V) \in C([0,T];\R^{2d})$ we have
\begin{equation*}
	|\cT_V(\omega,X,V)(t)| \le |V_0(\omega)|+T\Norm{F}{L^\infty([0,T]\times \R^{2d})}+\sqrt{2\sigma} M_B(\omega)=:M_V(\omega).
\end{equation*}
Furthermore,
\begin{equation*}
	|\cT_X(\omega,X,V)(t)| \le |X_0(\omega)|+TM_V(\omega)=:M_X(\omega).
\end{equation*}
Hence \eqref{eq:step2} holds with  $M(\omega)=\sqrt{(M_V(\omega))^2+(M_X(\omega))^2}$.

Now we prove the equi-H\"older condition, i.e. we show that for any $\omega \in \Omega$ and for any $\gamma<\sfrac{1}{2}$ there exists a constant $L_\gamma(\omega)$ such that, for all $(X,V) \in C([0,T];\R^{2d})$ and $t,s \in [0,T]$,
\begin{equation}\label{eq:step3}
	|\cT(\omega,X,V)(t)-\cT(\omega,X,V)(s)| \le L_\gamma(\omega)|t-s|^{\gamma}.
\end{equation}
To do this, let 
\begin{equation}\label{eq:LBgamma}
	L_{B,\gamma}(\omega)=\sup_{\substack{t,s \in [0,T] \\ t \not = s}}\frac{|B(\omega,t)-B(\omega,s)|}{|t-s|^{\gamma}},
\end{equation}
which is finite since $B$ is locally H\"older continuous of order $\gamma<\sfrac{1}{2}$ (see \cite[Theorem I.2.2]{revuzyor}). Hence, for any $t,s \in [0,T]$ and $(X,V) \in C([0,T];\R^{2d})$
\begin{equation*}
	|\cT_V(\omega,X,V)(t)-\cT_V(\omega,X,V)(s)| \le \Norm{F}{L^\infty([0,T]\times \R^{2d})}|t-s|+\sqrt{2\sigma} L_{B,\gamma}(\omega)|t-s|^{\gamma}.
\end{equation*}
We set
\begin{equation*}
	L_{V,\gamma}(\omega)=\sqrt{2\sigma}L_{B,\gamma}(\omega)+\Norm{F}{L^\infty([0,T]\times \R^{2d})}T^{1-\gamma}.
\end{equation*}	
Next, observe that
\begin{equation*}
	|\cT_X(\omega,X,V)(t)-\cT_X(\omega,X,V)(s)| \le M_V(\omega)|t-s| \le M_V(\omega)T^{1-\gamma}|t-s|^{\gamma}=:L_{X,\gamma}(\omega)|t-s|^{\gamma}.
\end{equation*}
We obtain \eqref{eq:step3} by setting $L_\gamma(\omega)=\sqrt{(L_{V,\gamma}(\omega))^2+(L_{X,\gamma}(\omega))^2}$.

Once we established that $\cT$ is a continuous and compact random operator, by the Bharucha-Reid Fixed Point Theorem \ref{thm:BRfp}, we know that there exists a stochastic process $(X,V) \in C([0,T];\R^{2d})$ such that $\cT(X,V)=(X,V)$, i.e. $(X,V)$ is a solution of \eqref{eq:SDEstrongaux-intro}. 

Now set $\bm{\mu}={\rm Law}(X,V)$. Let us show that $\bm{\mu} \in \W_1(C([0,T];\R^{2d}))$. To do this, recall that
\begin{equation*}
	V(t)=V_0+\int_0^tF(s,X(s),V(s))\, ds+\sqrt{2\sigma}B(t)
\end{equation*}
and then
\begin{equation}\label{eq:Vtpresup}
	|V(t)| \le |V_0|+\Norm{F}{L^\infty([0,T]\times \R^{2d})}T+\sqrt{2\sigma} M_B.
\end{equation}
By Doob's maximal inequality (see \cite[Theorem II.1.7]{revuzyor}) we have
\begin{equation}\label{eq:expMB}
	\E\left[M_B\right] \le \sqrt{\E\left[M_B^2\right]} \le 2\sqrt{T},
\end{equation}
hence, taking the supremum over $[0,T]$ and the expectation in \eqref{eq:Vtpresup}, we get
\begin{equation}\label{eq:Vtpresup2}
	\E\left[\sup_{t \in [0,T]}|V(t)|\right] \le M_1(\mu_0)+\Norm{F}{L^\infty([0,T]\times \R^{2d})}T+ \sqrt{8\sigma T}=:H_{1,V}(\mu_0,T,F).
\end{equation}
Furthermore, we have
\begin{equation*}
	\sup_{t \in [0,T]}|X(t)| \le |X_0|+T\sup_{t \in [0,T]}|V(t)|
\end{equation*}
so that, taking the expectation,
\begin{equation}\label{eq:Xtpresup2}
	\E\left[\sup_{t \in [0,T]}|X(t)|\right] \le M_1(\mu_0)+TH_{1,V}(\mu_0,T,F)=:H_{1,X}(\mu_0,T,F).
\end{equation}
Setting $H_{1}(\mu_0,T,F)=H_{1,X}(\mu_0,T,F)+H_{1,V}(\mu_0,T,F)$ we have 


\begin{equation}\label{eq:supboundcouple}
	\E\left[\sup_{t \in [0,T]}|(X(t),V(t))|\right] \le \E\left[\sup_{t \in [0,T]}|X(t)|+\sup_{t \in [0,T]}|V(t)|\right] \le H_1(\mu_0,T,F),
\end{equation}
that proves $\bm{\mu} \in \W_1(C([0,T];\R^{2d}))$. By Proposition \ref{prop:Ito} we know that $\bm{\mu}$ solves \eqref{probmu}. In particular, by Proposition \ref{prop:uniqueness}, $\bm{\mu}$ is the unique solution of \eqref{probmu}.


Hence, in particular, if we consider two solutions $(X_1,V_1)$ and $(X_2,V_2)$ with laws respectively $\bm{\mu}^1$ and $\bm{\mu}^2$, then both flows of probability measures belong to $C([0,T];\W_1(\R^{2d}))$ and solve \eqref{probmu}, hence $\bm{\mu}^1=\bm{\mu}^2$ and uniqueness in law holds. Since we have both strong existence and uniqueness in law, by a well-known result by Cherny (see \cite[Theorem 3.2]{Cherny}), we get that the strong solution of \eqref{eq:SDEstrongaux-intro} is pathwise unique. Once we established existence of local solutions and pathwise uniqueness, the fact that $T>0$ is arbitrary guarantees that the solution is global.


\textbf{Case 2: $F$ is unbounded.} Consider a sequence of cut-off functions $(\eta_n)_{n \in \N} \subset C^\infty_c(\R^{2d})$ with support on the ball $B_{n+1}$ of $\R^{2d}$, $\eta_n \equiv 1$ on $B_n$ and $\eta_n(x) \in [0,1]$. Let $F_n(s,Z)=F(s,Z)\eta_n(Z)$ and consider the sequence $Z_n=(X_n,V_n)$ of strong solutions of \eqref{eq:SDEstrongaux-intro} with nonlinearity $F_n$. Define the sequence of stopping times
\begin{equation*}
	\tau_n:=\inf\{t>0: \ |Z_n(t)|>n\}.	
\end{equation*}
First, let us prove that this sequence of stopping times is almost surely non-decreasing. To do this, consider $m \le n$ and define the events
\begin{equation*}
	E_m:=\{(X_0,V_0) \in B_m\}, \qquad E_m^c:=\Omega \setminus E_m.
\end{equation*}
Consider then
\begin{equation*}
	\bP(\tau_m \le \tau_n)=\bP(\tau_m \le \tau_n | E_m)\bP(E_m)+\bP(\tau_m \le \tau_n | E_m^c)\bP(E_m^c).
\end{equation*}
First observe that on $E_m^c$, clearly, $\tau_m=0$. Hence
\begin{equation*}
	\bP(\tau_m \le \tau_n | E_m^c)=1.
\end{equation*}
We only need to prove that
\begin{equation}\label{eq:cond1}
	\bP(\tau_m \le \tau_n | E_m)=1.
\end{equation}
To do this, consider the probability space $(\Omega, \mathcal{F}, \bP(\cdot \mid E_m))$. Then, since $B$ is still a Brownian motion in this probability space, $(X_m,V_m)$ and $(X_n,V_n)$ are still strong solutions of \eqref{eq:SDEstrongaux-intro} with nonlinearity $F_m$ and $F_n$. By continuity we know that $\tau_m,\tau_n>0$ almost surely. Furthermore, for $t \in [0,\min\{\tau_m,\tau_n\})$  $(X_m,V_m)$ and $(X_n,V_n)$ both satisfy \eqref{eq:SDEstrongaux-intro} with non-linearity $F$, since $(X_j(t),V_j(t)) \in B_j$ for all $t \in [0,\min\{\tau_m,\tau_n\})$, $j=m,n$. Hence, by pathwise uniqueness, the two solutions must coincide up to $\min\{\tau_m,\tau_n\}$. However, by continuity, we get
$|(X_n(\tau_m),V_n(\tau_m))|=|(X_m(\tau_m),V_m(\tau_m))|=m$, that implies $(X_n(\tau_m),V_n(\tau_m)) \in B_n$. Thus $\tau_n \ge \tau_m$ a.s. given $E_m$ and then \eqref{eq:cond1} follows. 

Since $(\tau_n)_{n \ge 1}$ is non-decreasing, we can consider the stopping time $\tau_\infty=\lim_{n \to \infty}\tau_n$. Now we want to prove that $\tau_\infty=\infty$ a.s. To do this, fix $n \in \N$ and observe that
\begin{align*}
	|V_n(t)| &\le |V_0|+\int_0^t |F(s,Z_n(s))||\eta_n(Z_n(s))|ds+\sqrt{2\sigma}|B(t)| \\
	&\le |V_0|+C(F)t+C(F)\int_0^t|X_n(s)|^\frac{\beta}{3}ds+C(F)\int_0^t|V_n(s)|^\beta ds+\sqrt{2\sigma}|B(t)|.
\end{align*}
Since $\beta \le 1$, we can apply Young's inequality with $p=\frac{3}{\beta}$ to the first integrand and with $p=\frac{1}{\beta}$ to the second integrand. This leads to
\begin{align*}
	|V_n(t)|\le |V_0|+C(F)\left(3-\frac{4\beta}{3}\right)t+C(F)\beta\int_0^t(|X_n(s)|+|V_n(s)|)ds+\sqrt{2\sigma}\max_{0 \le s \le t}|B(s)|.
\end{align*}
Furthermore, we clearly have
\begin{equation*}
	|X_n(t)| \le |X_0|+\int_0^t|V_n(s)|ds,
\end{equation*}
hence
\begin{align*}
	|Z_n(t)|\le C\left(|Z_0|+t+\int_0^t|Z_n(s)|ds+\max_{0 \le s \le t}|B(s)|\right),
\end{align*}
where the constant $C$ only depends on $F$, $\beta$ and $\sigma$. By Gr\"onwall inequality, we get
\begin{align}\label{eq:Gronwall2}
	|Z_n(t)| \le C\left(|Z_0|+t+\max_{0 \le s \le t}|B(s)|\right)e^{Ct}.
\end{align}
In particular, for $t=\tau_n$,
\begin{align*}
	n \le C\left(|Z_0|+\tau_n+\max_{0 \le s \le \tau_n}|B(s)|\right)e^{C\tau_n}.
\end{align*}
Now assume by contradiction that $\bP(\tau_\infty<\infty)>0$ and consider any $\omega$ such that $\tau_\infty(\omega)<\infty$ and $s \in \R_0^+ \mapsto B(s,\omega)$ is continuous. For such a fixed $\omega$, taking the limit as $n \to +\infty$, we achieve
\begin{align*}
	\infty \le C\left(|Z_0|+\tau_\infty(\omega)+\max_{0 \le s \le \tau_\infty(\omega)}|B(s,\omega)|\right)e^{C\tau_\infty(\omega)}<\infty,
\end{align*}
that is absurd. Hence $\tau_\infty=\infty$ almost surely. Set also, for convenience, $\tau_0=0$. Now define, for any $t \ge 0$
\begin{equation*}
	N(t)=\min\{n \in \N: \ t \le \tau_n\},
\end{equation*}
that is well-defined since $\tau_n \to \infty$ almost surely. Recall that, by pathwise uniqueness, we have $(X_m(t),V_m(t))=(X_n(t),V_n(t))$ for $t \in [0,\tau_m]$ whenever $m \le n$. Hence, the process
\begin{equation*}
	X(t)=X_{N(t)}(t), \ V(t)=V_{N(t)}(t), \ \forall \, t \ge 0
\end{equation*}
is well-defined. It remains to prove that $(X,V)$ is a solution to \eqref{eq:SDEstrongaux-intro}. Indeed, for any $t \ge 0$ we have
\begin{align*}
	V(t)&=V_{N(t)}(t)=V_0+\int_0^t F_{N(t)}(s,Z_{N(t)}(s))ds+\sigma B(t)\\
	&=V_0+\int_0^t F(s,Z_{N(s)}(s))ds+\sigma B(t)\\
	&=V_0+\int_0^t F(s,Z(s))ds+\sigma B(t).
\end{align*}
This proves that $(X,V)$ is the unique strong solution of \eqref{eq:SDEstrongaux-intro}. Now observe that, for fixed $T>0$, arguing as we did for \eqref{eq:Gronwall2}, we have
\begin{equation*}
	\sup_{t \in [0,T]}|Z(t)| \le C\left(|Z_0|+T+M_B\right)e^{CT}
\end{equation*}
and taking the expectation we have
\begin{equation*}
	\E\left[\sup_{t \in [0,T]}|Z(t)|\right] \le C\left(\E[|Z_0|]+T+\sqrt{T}\right)e^{CT}<\infty.
\end{equation*}
Hence, if we set $\bm{\mu}:={\rm Law}(X,V)$, we have $\bm{\mu} \in \W_1(C([0,T];\R^{2d}))$. By Proposition \ref{prop:Ito}, this guarantees that $\bm{\mu}$ solves \eqref{probmu} and then by Proposition \ref{prop:uniqueness} we know that it is the unique solution.
\end{proof}



\newpage

\appendix
\section{Group convolutions} \label{app:A}
Now, we recall the definition of family of mollifiers and 
of convolution with respect to the Lie group $\mathbb{K}$. 
First of all, we consider a function $\rho \in C^{\infty}_{\rm c}(\mathbb{R}^{1+ 2d})$ such that 
\begin{enumerate}
	\item $\rho \ge 0$ in $\mathbb{R}^{1+2d}$; 
	\item supp$\rho \subseteq {\sf B}_1(\sfrac32,0,0)$, where ${\sf B}_1(\sfrac32,0,0)$ is the $d_{\mathbb{K}}$-ball of radius $1$ and center at the point $(\frac32,0,0)$; 
	\item $\int_{\mathbb{R}^{1+2d}} \rho(\tau, \zeta) \, d \tau \, d\zeta = 
	\int_{{\sf B}_1(\sfrac32,0,0)} \rho(\tau, \zeta) \, d\tau \, d\zeta =1$.
\end{enumerate}
Then, for every~$\varepsilon>0$ we define the convolution kernel 
\begin{equation} \label{mollifier}
	\rho_\varepsilon(\tau, \zeta) := \varepsilon^{-4d-2} \rho \big( \Phi_{1 / \varepsilon} (\tau, \zeta) \big).
\end{equation}	 
Moreover, if we consider a function $f \in C([0,+\infty) \times \mathbb{R}^{2d})$, then the mollifier of $f$ related to $\rho_\varepsilon$ is
a function $f_\varepsilon : (\sfrac52 \, \varepsilon^2, + \infty) \times \mathbb{R}^{2d} \rightarrow \mathbb{R}$ defined as 
\begin{align}\label{convolution}
	f_{\varepsilon} (t,z) := \rho_\varepsilon \ast f(t,z)
	&:= \int _{[0, + \infty) \times \mathbb{R}^{2d}}  
	\rho_\varepsilon \left( (t,z) \circ (\tau, \zeta)^{-1} \right) f(\tau, \zeta) \, d \tau \, d \zeta \\ \nonumber
	& = \int _{{\sf B}_1(\sfrac32,0,0)}  
	\rho(\tau, \zeta) f \left( \Phi_\varepsilon (\tau, \zeta )^{-1} \circ (t,z)\right) \, d\tau \, d \zeta .
\end{align} 
This definition is well-posed since 
\begin{enumerate}
	\item for every fixed $(t,z) \in (\sfrac52 \, \varepsilon^2,+\infty) \times \mathbb{R}^{2d}$ we have
	\begin{equation*}
		{\rm supp} \left( (\tau, \zeta) \mapsto \rho_\varepsilon \left( (t,z) \circ (\tau, \zeta)^{-1} \right) \right)
		\subseteq [0,+\infty) \times \mathbb{R}^{2d};		
	\end{equation*}	
	\item for every $(\zeta, \tau) \in {\sf B}_1(\sfrac32,0,0)$ and $(t, z) \in (\sfrac52\,  \varepsilon^2, + \infty) \times \mathbb{R}^{2d}$ there holds
	\begin{align*}
		\left( \Phi_\varepsilon (\tau, \zeta)^{-1} \circ (t,z) \right) \in (0, + \infty) \times \mathbb{R}^{2d}.
	\end{align*} 
\end{enumerate}

Additionally, given $\rho \in C^{\infty}_0(\mathbb{R}^{1+2d})$, by a standard dominated convergence argument we may infer 
$f_\varepsilon \in C^{\infty}((\sfrac52 \, \varepsilon^2, + \infty) \times \mathbb{R}^{2d})$. Firstly, we begin observing that the regularization procedure does not change the upper bounds of the involved function.

\begin{lem}\label{lem:sublinearafterconv}
	Let $f:[0,T] \times \R^{2d} \to \R$ be a Caratheodory function such that for some $C>0$, with $\beta \in (0,1)$,
	$$
	|f(t,z)| \le C \left( 1 + \|z\|_B^\beta \right) \quad \forall z \in \R^{2d},
	$$
	where $\| \cdot \|_B$ is defined in \eqref{groupnorm}, 
	and set $f_\varepsilon=\rho_\varepsilon \ast f$. Then there exists a constant $\bar{C}>0$ such that 
	$$
	|f_\varepsilon(t,z)| \le \bar{C} \left( 1 + \|z\|_B^\beta \right) \quad \forall z \in \R^{2d}.
	$$
\end{lem}
\begin{proof}
	Just observe that for any $(t,z) \in \R^{1+2d}$ , by the above assumptions we have
	\begin{align*}
		\left|f_\varepsilon(t,z)\right| 
		&\le \int_{[0,+\infty) \times \R^{2d}}\rho_{\varepsilon}\left((t,z)\circ (\tau,\zeta)^{-1}\right)\left|f(\tau,\zeta)\right|d\tau \, d\zeta \\
		&\le C \left( 1 + \int_{\Gamma_\varepsilon(t,x,v)}\rho_{\varepsilon}\left((t,z)\circ (\tau,\zeta)^{-1}\right)  \| \zeta \|_B^\beta  d\tau \, d\zeta \right),
	\end{align*}
	where we used the fact that $\Norm{\rho}{L^1([0,+\infty) \times \R^{2d}}=1$ and we denoted by
	$\Gamma_\varepsilon(t,x,v)$ the support of $(\tau,\zeta) \mapsto \rho_\varepsilon((t,z) \circ (\tau,\zeta)^{-1})$.
	
	Next, we observe there exists a constant $\mathfrak{C}>0$ such that 
	$\|\zeta\|_B \le \|z\|_B+\mathfrak{C}\varepsilon \le \|z\|_B+\mathfrak{C}$ and hence, recalling $\beta \in (0,1)$, we have
	\begin{align*}
		|f_\varepsilon(t,z)|& \le \mathfrak{C} \left(1+\| z \|_B + \mathfrak{C} \right) \le \bar{C}(1+\| z \|^\beta_B),
	\end{align*}
	where we also used the fact that $\Norm{\rho_\varepsilon}{L^1(\R \times \R^{2d})}=1$.
\end{proof}

We also recall some estimates involving the derivatives of regularized functions.
\begin{lem}
	\label{conv-dert}
	Let $f \in C([0,\infty) \times \R^{2d}) \cap L^1([0,\infty) \times \R^{2d})$ and let $\{\rho_\varepsilon\}_{\varepsilon>0}$ be a family of mollifiers as defined above. 
	Set also $f_\varepsilon:=\rho_\varepsilon \ast f$. Then for every $\varepsilon >0$ and for every $(t,z) = (t,x,v) \in \left(\frac{5}{2}\varepsilon^2,+\infty\right)\times \R^{2d}$ it holds
	\begin{equation*}
		|\partial_t f_\varepsilon(t,z)| \le \frac{C}{\varepsilon^{4d+4}} \| f \|_{L^1\left( [0,\infty) \times \R^{2d} \right) }, \quad |\partial_{x_i} f_\varepsilon(t,z)| \le \frac{C}{\varepsilon^{4d+5}} \| f \|_{L^1\left( [0,\infty) \times \R^{2d} \right) } \, \, \, \forall i =1, \ldots, d,
	\end{equation*}
	where $C>0$ is a suitable constant only depending on $\rho$.
	
	Furthermore, if there exists $T>0$ such that $f(t,z)=0$ for any $t \ge T$ and $z \in \R^{2d}$, then, for $\varepsilon \in (0,1)$ and $(t,z) = (t,x,v) \in \left(\frac{5}{2}\varepsilon^2,+\infty\right) \times \R^{2d}$,
	\begin{equation*}
		|\partial_{v_i} f_\varepsilon(t,z)| \le \frac{CT}{\varepsilon^{4d+3}} \| f \|_{L^1\left( [0,T] \times \R^{2d} \right) }, \qquad |\partial_{v_i \, v_j} f_\varepsilon(t,z)| \le \frac{CT^2}{\varepsilon^{4d+4}} \| f \|_{L^1\left( [0,T] \times \R^{2d} \right) },
	\end{equation*}
	for every $i,j = 1, \ldots, d$ and where $C>0$ is a constant depending on $\rho$.
\end{lem}
\begin{proof}
	Let us fix $(t,z) \in (\sfrac52 \, \varepsilon^2,+\infty) \times \mathbb{R}^{2d}$, 
	then by definition \eqref{convolution} of convolution, for every $\varepsilon >0$ we have
	\begin{align*}
		|\partial_t f_{\varepsilon} (t,z)| 
		&:= \left|\partial_t \int_{0}^{+\infty}  \int _{ \mathbb{R}^{2d}}  
		\rho_\varepsilon \left( (t,z) \circ (\tau, \zeta)^{-1} \right) f(\tau, \zeta) \, d \tau \, d \zeta \right|\\ \nonumber
		&= \frac{1}{\varepsilon^{4d+4}} \, \left|\int_{0}^{+\infty}  \int _{ \mathbb{R}^{2d}}  
		\partial_t \rho \left( \Phi_{\frac{1}{\varepsilon}} \left( (t,z) \circ (\tau, \zeta)^{-1} \right) \right) f(\tau, \zeta) \, d \tau \, d \zeta \right|\\
		&\le \frac{C}{ \varepsilon^{4d+4} } \| f \|_{L^1( [0, + \infty) \times \mathbb{R}^{2d} )}.	
	\end{align*}
	where $(t,z)=(t,x,v)$ and $(\tau, \zeta)= (\tau, \xi, \eta)$, we applied Lebesgue Theorem and estimated the absolute value 
	on the right-hand side. 
	
	The argument for the second inequality is the same. In order to handle the third inequality, we observe that
	
	\begin{align*}
		|\partial_{v_i} f_{\varepsilon} (t,z)| 
		&:= \left|\partial_{v_i} \int_{0}^{+\infty}  \int _{ \mathbb{R}^{2d}}  
		\rho_\varepsilon \left( (t,z) \circ (\tau, \zeta)^{-1} \right) f(\tau, \zeta) \, d \tau \, d \zeta \right|\\ \nonumber
		&=  \frac{1}{\varepsilon^{4d+2}} \, \left|\partial_{v_i}  \int_{0}^{+\infty}  \int _{ \mathbb{R}^{2d}}  
		\rho \left( \Phi_{\frac{1}{\varepsilon}} \left( (t,z) \circ (\tau, \zeta)^{-1} \right) \right) f(\tau, \zeta) \, d \tau \, d \zeta\right| \\ \nonumber
		&= \frac{1}{\varepsilon^{4d+2}} \, \left(\frac{1}{\varepsilon^3}\left|\int_{0}^{T}  \int _{ \mathbb{R}^{2d}}  
		\tau \partial_{v_i} \rho \left( \Phi_{\frac{1}{\varepsilon}} \left( (t,z) \circ (\tau, \zeta)^{-1} \right) \right) f(\tau, \zeta) \, d \tau \, d \zeta \right|\right.\\
		&\left.+\frac{1}{\varepsilon}\left|\int_{0}^{T}  \int _{ \mathbb{R}^{2d}}  
		\partial_{v_i} \rho \left( \Phi_{\frac{1}{\varepsilon}} \left( (t,z) \circ (\tau, \zeta)^{-1} \right) \right) f(\tau, \zeta) \, d \tau \, d \zeta \right|\right)\\
		&\le \frac{C T}{ \varepsilon^{4d+5} } \| f \|_{L^1( [0, T] \times \mathbb{R}^{2d} )}.	
	\end{align*}
	The fourth inequality can be carried out analogously.
\end{proof}

Finally, we prove the following approximation result of the integral of functions against narrowly continuous flows of probability measures.
\begin{lem}\label{lem:L1Cara}
	Let $f:[0,+\infty) \times \R^{2d} \to \R$ be a Carathéodory function such that for a certain $ \beta \in (0,1)$
	$$
	|f(t,z)| \le C \left( 1 + \| z \|_B^\beta \right) \qquad \forall z \in \R^{2d},
	$$ 
	and let us denote $f_\varepsilon=\rho_\varepsilon \ast f$. 
	Furthermore, let
	$G:[0,+\infty) \times \R^{2d} \to \R$ be bounded Caratheodory function and $\{\mu_{t}\}_{t \in [0,T]} \in C([0,T];\cP(\R^{2d}))$. 
	Then, as $\varepsilon \to 0$, for all $t>0$,
	\begin{equation*}
		\int_t^T \int_{\R^{2d}}G(s,z)\cdot f_\varepsilon(s,z) d\mu_s(z)\, ds \to \int_t^T \int_{\R^{2d}}G(s,z) \cdot f(s,z) d\mu_s(z) \, ds.
	\end{equation*}
\end{lem}

\begin{proof}
	First of all, we apply Scorza Dragoni's Theorem to the function $f$ and
	for a fixed $\delta>0$ and $T>0$ we find a function $g_\delta: [0,T] \times \R^{2d} \to \R$ that is continuous and such that there exists a compact 
	$K \subset [0,T]$ with $g_\delta=f$ on $K \times \R^{2d}$ and $|[0,T]\setminus K|<\delta$, where $| \cdot |$ denotes the Lebesgue measure of the set. 
	
	Notice that, without loss of generality, by truncation we can assume that 
	$$
	|g_\delta(s,z)|, |f(s,z)| \le C ( 1 + \|z \|_B^{\beta}).
	$$ 
	Then, we have the pointwise estimate
	\begin{equation*}
		|f(s,z)-g_\delta(s,z)|=2C  \left( 1 + \|z \|_B^{\beta} \right) \, \chi_{[0,T] \setminus K} (s),
	\end{equation*}
	where $\chi_{[0,T] \setminus K}$ is the characteristic function of the set $[0,T]\setminus K$.
	
	Now, if we denote by $(\tau,\zeta)=(\tau,y,w) \in \R^{1+2d}$, we observe that
	\begin{align*}
		|\rho_\varepsilon \ast (f-g_\delta)(s,z)| &\le \int_{0}^{\infty}\int_{\R^{2d}} |\rho_\varepsilon((s,z) \circ (\tau,\zeta)^{-1})||f(\tau,\zeta)-g_\delta(\tau,\zeta)| d\zeta \, d\tau\\
		& \le 2C\int_{0}^{\infty}\int_{\R^{2d}} \rho_\varepsilon((s,z) \circ (\tau,\zeta)^{-1}) \left( 1 + \| \zeta \|_B^{\beta} \right)\chi_{[0,T]\setminus K}(\tau)d\tau \, d\zeta \\
		&= 2C\int_{0}^{\infty}\int_{\R^{2d}} \rho_\varepsilon((s,z) \circ (\tau,\zeta)^{-1}) \chi_{[0,T]\setminus K}(\tau)d\tau \, d\zeta  \\
		&\qquad + 2C\int_{0}^{\infty}\int_{\R^{2d}} \rho_\varepsilon((s,z) \circ (\tau,\zeta)^{-1})  \| \zeta \|_B^{\beta} \chi_{[0,T]\setminus K}(\tau)d\tau \, d\zeta\\
		&=:J_1(s,z)+J_2(s,z)
	\end{align*}
	Now, by consideringthe definition of group traslation \eqref{group_law}, of group convolution \eqref{convolution} and performing a change of variables, we set
	\begin{equation*}
		\widetilde{\rho}_\varepsilon(\tau-s)=\int_{\R^{2d}}\rho_\varepsilon((s,z)\circ (\tau,\zeta)^{-1})d\zeta=\frac{1}{\varepsilon^2}\int_{\R^{2d}}\rho(-\varepsilon^{-2}(\tau-s),\zeta)d\zeta,	
	\end{equation*}
	and then we have
	\begin{align*}
		I_1 (z):=  \int_t^T J_1(s,z) \, ds 
		\le 2C \int_t^T\int_{0}^{\infty}\widetilde{\rho}_{\varepsilon}(\tau-s)\chi_{[0,T]\setminus K}(s) ds \le 2C\delta
	\end{align*}
	by Young's convolution inequality applied to the Euclidean one-dimensional convolution. 
	Analogously, by defining 
	\begin{align*}
		I_2 (z):= \int_t^T J_2(s,z) \, ds  \le  2C\delta ( 1+ \| z \|_B^{\beta}),
	\end{align*}
	where we proceed as in the proof of Lemma \ref{lem:sublinearafterconv}.
	We deduce that	
	\begin{align*}
		&\left|\int_t^T \int_{\R^{2d}}G(s,z)\cdot f_\varepsilon(s,z) d\mu_s(z)ds- \int_t^T \int_{\R^{2d}}G(s,z)\cdot f(s,z) d\mu_s(z)ds\right| \\
		&\qquad \le \int_t^T \int_{\R^{2d}}|G(s,z)| |(\rho_\varepsilon \ast (f-g_\delta))(s,z)| d\mu_s(z)ds\\
		&\qquad +\left|\int_t^T \int_{\R^{2d}}G(s,z)\cdot (\rho_\varepsilon \ast g_\delta)(s,z) d\mu_s(z)ds-\int_t^T \int_{\R^{2d}} G(s,z)\cdot g_\delta(s,z) d\mu_s(z)ds\right|\\
		&\qquad+\int_t^T \int_{\R^{2d}} |G(s,z)| |f(s,z)-g_\delta(s,z)| d\mu_s(z)ds\\
		&\qquad \le 6C^2 \delta + 4C^2\delta \overline{M}_1 ({\bm \mu};T) \\
		&\qquad \qquad + C \int_t^T \int_{\R^{2d}} \Big | (\rho_\varepsilon \ast g_\delta)(s,z) d\mu_s(z)ds- g_\delta(s,z) \Big | d\mu_s(z) ds \\
	\end{align*}
	Taking the limit supremum as $\varepsilon \to 0$ the integral term vanishes as $g_\delta$ is continuous in $[0,+\infty)\times \R^{2d}$ and bounded, so that
	\begin{align*}
		\limsup_{\varepsilon \to 0} &\left|\int_0^T \int_{\R^{2d}}G(s,z)\cdot f_\varepsilon(s,z) d\mu_t(z)ds- \int_0^T \int_{\R^{2d}}G(s,z)\cdot f(s,z) d\mu_t(z)ds\right| \\
		&\qquad \qquad \qquad \qquad \le 6C^2 \delta + 4C^2\delta \overline{M}_1 ({\bm \mu};T)
	\end{align*}
	Finally, send $\delta \to 0$ to end the proof.
\end{proof}

It is worth mentioning that the regularization procedure introduced in \eqref{convolution} commutes with the Lie derivative just introduced.
\begin{lem}\label{lemY}
	Let $f \in C([0,\infty) \times \R^{2d})$ with $Yf \in L^1_{\rm loc}([0,\infty) \times \mathbb{R}^{2d})$, let $\{\rho_\varepsilon\}_{\varepsilon>0}$ be a family of mollifiers as defined in \eqref{mollifier} and set $f_\varepsilon:=\rho_\varepsilon \ast f$. Then $Yf_\varepsilon=\rho_\varepsilon \ast Yf$. 
\end{lem}
\begin{proof}
	Let us fix $(t,z) \in (0,+\infty)\times \R^{2d}$. Then we observe $Y$ is a left-invariant vector field with respect to the composition law $\circ$. 
	Hence, we are allowed to write
	\begin{align*}
		Y f_\epsilon (t,z) = \int_{{\sf B}_1(\sfrac32,0,0)} \rho (\tau, \zeta) (Y f ) \left( \Phi_\epsilon (\tau, \zeta)^{-1} \circ (t,z) \right) \, d\tau \, d \zeta.
	\end{align*}	
\end{proof}

\bibliographystyle{abbrv}

\begin{thebibliography}{50}

\bibitem{Almi20231}
S.~Almi, M.~Morandotti, and F.~Solombrino.
\newblock Optimal control problems in transport dynamics with additive noise.
\newblock {\em Journal of Differential Equations}, 373:1 – 47, 2023.

\bibitem{ambrosio2005gradient}
L.~Ambrosio, N.~Gigli, and G.~Savar{\'e}.
\newblock {\em Gradient flows: in metric spaces and in the space of probability
  measures}.
\newblock Springer Science \& Business Media, 2005.

\bibitem{AP-survey}
F.~Anceschi and S.~Polidoro.
\newblock A survey on the classical theory for {Kolmogorov} equation.
\newblock {\em Matematiche}, 75(1):221--258, 2020.

\bibitem{arendt2011vector}
W.~Arendt, C.~J.~K. Batty, M.~Hieber, and F.~Neubrander.
\newblock {\em Vector-valued Laplace Transforms and Cauchy Problems},
  volume~96.
\newblock Springer Science \& Business Media, 2011.

\bibitem{Ascione20236965}
G.~Ascione, D.~Castorina, and F.~Solombrino.
\newblock Mean-field sparse optimal control of systems with additive white
  noise.
\newblock {\em SIAM Journal on Mathematical Analysis}, 55(6):6965 – 6990,
  2023.

\bibitem{AMP23} 
G.~Ascione, Y~Mishura, and E~Pirozzi.
\newblock \emph{Fractional Deterministic and Stochastic Calculus.}
\newblock Walter de Gruyter GmbH \& Co KG, 2023.

\bibitem{BiagiBramanti}
S.~Biagi and M.~Bramanti.
\newblock Schauder estimates for Kolmogorov-Fokker-Planck operators with coefficients measurable in time and H\"older continuous in space.
\newblock {\em Journal of Mathematical Analysis and Applications}, 533(1):127996, 2024.

\bibitem{BLU}
A.~Bonfiglioli, E.~Lanconelli, and F.~Uguzzoni.
\newblock {\em Stratified {Lie} groups and potential theory for their
  sub-{Laplacians}}.
\newblock Springer Monogr. Math. New York, NY: Springer, 2007.

\bibitem{Cherny} A.~S.~Cherny.
\newblock On the uniqueness in law and the pathwise uniqueness for stochastic differential equations.
\newblock {\em Theory of Probability \& Its Applications} 46.3 (2002): 406-419.



\bibitem{DeckKruse}
T.~Deck and S.~Kruse.
\newblock Parabolic differential equations with unbounded coefficients - a
  generalization of the parametrix method.
\newblock {\em Acta Applicandae Mathematicae}, 74:71--91, 2002.

\bibitem{francesco2005class}
M.~Di~Francesco and A.~Pascucci.
\newblock On a class of degenerate parabolic equations of kolmogorov type.
\newblock {\em Applied Mathematics Research eXpress}, 2005(3):77--116, 2005.

\bibitem{Flandoli11}
F.~Flandoli
\newblock {\em Random Perturbation of PDEs and Fluid Dynamic Models: École d’été de Probabilités de Saint-Flour XL–2010.}
\newblock Springer Science \& Business Media, 2011.

\bibitem{itoh1979random}
S.~Itoh.
\newblock Random fixed point theorems with an application to random
  differential equations in {B}anach spaces.
\newblock {\em Journal of Mathematical Analysis and Applications},
  67(2):261--273, 1979.

\bibitem{LP}
E.~Lanconelli and S.~Polidoro.
\newblock On a class of hypoelliptic evolution operators.
\newblock {\em Rendiconti del Seminario Matematico di Torino}, 52(1):29--63, 1994.

\bibitem{lucertinigood}
G.~Lucertini, S.~Pagliarani, and A.~Pascucci.
\newblock Optimal regularity for degenerate Kolmogorov equations
in non-divergence form with rough-in-time coefficients.
\newblock {\em Journal of Evolution Equations}, 23:69, 2023.

\bibitem{oksendal2013stochastic}
B.~Oksendal.
\newblock {\em Stochastic differential equations: an introduction with
  applications}.
\newblock Springer Science \& Business Media, 2013.

\bibitem{Orlando2023}
G.~Orlando.
\newblock Mean-field optimal control in a multi-agent interaction model for
  prevention of maritime crime.
\newblock {\em Advances in Continuous and Discrete Models}, 2023(1), 2023.

\bibitem{pascucci2011pde}
A.~Pascucci.
\newblock {\em PDE and martingale methods in option pricing}.
\newblock Springer Science \& Business Media, 2011.

\bibitem{polidoro}
S.~Polidoro.
\newblock  On a class of ultraparabolic operators of Kolmogorov-Fokker-Planck type. 
\newblock {\em Matematiche}, 49 (1994), pp. 53–105.

\bibitem{revuzyor}
D.~Revuz and M.~Yor.
\newblock {\em Continuous martingales and {B}rownian motion}.
\newblock Springer-Verlag, Berlin, 1999.

\bibitem{villani}
C.~Villani.
\newblock {\em Optimal transport}.
\newblock Springer-Verlag, Berlin, 2009.

\end{thebibliography}

\end{document}